\documentclass[a4paper,11pt]{article}

\usepackage{amsfonts,amsmath,amssymb,amsthm,amscd,url}
\usepackage[final]{showkeys}
\usepackage{xspace}
\usepackage[matrix,arrow]{xy}
\usepackage{mathrsfs,eucal,enumerate,mathtools}
\usepackage{epsfig,psfrag}
\usepackage{graphicx}
\usepackage{color}

\newtheorem{theorem}{Theorem}
\newtheorem{lemma}[theorem]{Lemma}
\newtheorem{proposition}[theorem]{Proposition}
\newtheorem{corollary}[theorem]{Corollary}

\newtheorem{definition}{Definition}[section]
\newtheorem{notation}[definition]{Notation}
\newtheorem{rem}{Remark}[section]

\newcommand{\eps}{\varepsilon}     
\newcommand{\la}{\lambda}
\newcommand{\defeq}{\coloneqq} 

\newcommand{\CC}{\mathbb{C}}       
\newcommand{\NN}{\mathbb{N}}  
\newcommand{\QQ}{\mathbb{Q}}       
\newcommand{\RR}{\mathbb{R}}       
\newcommand{\ZZ}{\mathbb{Z}}       
\newcommand{\BB}{\mathbb{B}}
\newcommand{\DD}{\mathbb{D}}

\renewcommand{\SS}{\mathbb{S}}
\newcommand{\TT}{\mathbb{T}}

\newcommand{\PP}{\widehat{\mathbb{C}}}

\newcommand{\id}{\mathrm{id}}

\newcommand{\ph}{\varphi}          
\newcommand{\al}{\alpha}
\newcommand{\be}{\beta}
\newcommand{\de}{\delta}
\newcommand{\De}{\Delta}
\newcommand{\ga}{\gamma}
\newcommand{\Ga}{\Gamma}
\newcommand{\ka}{\kappa}
\newcommand{\sig}{\sigma}
\renewcommand{\th}{\theta}
\newcommand{\Sig}{\Sigma}
\newcommand{\ze}{\zeta}
\newcommand{\om}{\omega}
\newcommand{\Om}{\Omega}

\newcommand{\cA}{\mathcal{A}}
\newcommand{\cE}{\mathcal{E}}
\newcommand{\cF}{\mathcal{F}}

\newcommand{\gC}{\mathscr C}       
\newcommand{\gF}{\mathscr F}       
\newcommand{\gH}{\mathscr H}       
\newcommand{\gK}{\mathscr K}
\newcommand{\lL}{\mathscr L}       
\newcommand{\gO}{\mathscr O}
\newcommand{\gM}{\mathscr M}
\newcommand{\fM}{\mathfrak M}

\newcommand{\col}{\colon}          

\newcommand{\I}{{\mathrm i}}
\newcommand{\dd}{{\mathrm d}}      
\newcommand{\ee}{{\mathrm e}}      
\newcommand{\pa}{\partial}
\newcommand{\na}{\nabla}
\newcommand{\ii}{^{-1}}
\newcommand{\demi}{\frac{1}{2}}
\newcommand{\ti}{\tilde}
\newcommand{\ens}{\enspace}
\newcommand{\IM}{\mathop{\Im m}\nolimits}
\newcommand{\RE}{\mathop{\Re e}\nolimits}
\newcommand{\ie}{\emph{i.e.}\ }
\newcommand{\eg}{\emph{e.g.}\ }
\newcommand{\resp}{\emph{resp.}\ }

\newcommand{\lhs}{{left-hand side}}
\newcommand{\rhs}{{right-hand side}}
\newcommand{\subcar}[1]{_{[#1]}}
\newcommand{\suppar}[1]{^{(#1)}}

\renewcommand{\lg}{\langle}
\newcommand{\rg}{\rangle}

\DeclareMathOperator{\conj}{conj}     
\DeclareMathOperator{\dist}{dist}     
    
\DeclareMathOperator{\range}{range}

\newcommand{\dst}{\displaystyle}

\newcommand{\wrt}{with respect to}

\newcommand{\ov}{\overline}

\newcommand{\mean}[1]{\lg #1 \rg}
\newcommand{\norm}[1]{\Vert#1\Vert}

\newcommand{\nor}[2]{\Vert#1\Vert_{\Chol(#2)}}
\newcommand{\Br}{\BB_\rho}
\newcommand{\Cr}{\gC_\rho}
\newcommand{\Crr}{\gC_{\rho,\RR}}
\newcommand{\Cn}[1]{\gC_{#1,\rho}}
\newcommand{\Cnr}[1]{\gC_{#1,\rho,\RR}}
\newcommand{\Cp}[1]{\gC^{+}_{#1,\rho}}
\newcommand{\Cpr}[1]{\gC^{+}_{#1,\rho,\RR}}
\newcommand{\Cm}[1]{\gC^{-}_{#1,\rho}}
\newcommand{\Cmr}[1]{\gC^{-}_{#1,\rho,\RR}}
\newcommand{\Cpm}[1]{\gC^{(\pm)}_{#1,\rho}}
\newcommand{\Cpmr}[1]{\gC^{(\pm)}_{#1,\rho,\RR}}
\newcommand{\nop}[2]{\Vert#1\Vert_{#2,+}}
\newcommand{\nom}[2]{\Vert#1\Vert_{#2,-}}
\newcommand{\nopm}[2]{\Vert#1\Vert_{#2,(\pm)}}

\newcommand{\ovci}[1]{\overset{\circ}{#1}}

\newcommand{\ovi}[1]{\ovci{\raisebox{0ex}[1.45ex]{$#1$}}}

\newcommand{\str}[1]{\!\raisebox{0ex}[#1ex]{\,}}
\newcommand{\IN}{^{\textnormal{(i)}}}
\newcommand{\EX}{^{\textnormal{(e)}}}
\newcommand{\Chol}{\gC^1_{\textrm{hol}}}

\newcommand{\HSrB}{H^\infty(S_r,B)}
\newcommand{\HSrBr}{H^\infty(S_r,\Br)}
\newcommand{\HSrpB}{H^\infty(S_{r'},B)}
\newcommand{\HSrp}{H^\infty(S_{r'})}

\newcommand{\eCe}{\eps\,\gC[[\eps]]}
\newcommand{\Ce}{\gC[[\eps]]}

\DeclareMathOperator{\DC}{DC}

\newcommand{\pth}{\partial_\theta}

\newcommand{\Igs}{Invariant graphs}
\newcommand{\ig}{invariant graph}
\newcommand{\igs}{invariant graphs}

\newcommand{\Ics}{Invariant graphs}

\begin{document}
\title{There is only one KAM curve}
\author{C. Carminati, S. Marmi, D. Sauzin}



\maketitle

\vspace{.6cm}

\begin{abstract}
We consider the standard family of area-preserving twist maps of the annulus and
the corresponding KAM curves.
Addressing a question raised by Kolmogorov, we show that, instead of viewing
these invariant curves as separate objects, each of which having its own
Diophantine frequency, one can encode them in a single function of the frequency
which is naturally defined in a complex domain containing the real Diophantine
frequencies and which is monogenic in the sense of Borel; this implies a
remarkable property of quasianalyticity, a form of uniqueness of the monogenic
continuation,
although real frequencies constitute a natural boundary for the analytic continuation
from the Weierstra\ss{} point of view because of the density of the resonances.
\end{abstract}

\vspace{.5cm}

\tableofcontents
\thispagestyle{empty}

\vfill\eject


\setcounter{section}{-1}


\section{Introduction}


In this article, we address what is, to our best knowledge, the oldest open
problem in KAM theory.
Indeed, in 1954, in his ICM conference \cite{Kol}, Kolmogorov asked whether the
regularity of the solutions of small divisor problems with respect to the
frequency could be investigated using appropriate analytical tools, suggesting a
connection with the theory of ``monogenic functions'' in the sense of \'Emile
Borel \cite{Bo}.\footnote{%
Towards the end of \cite{Kol}, he considers the example of a real analytic vector
field on the two-dimensional torus which depends analytically on a real
parameter~$\th$: he claims that, if the ratio of mean frequencies is not
constant, then there is a full measure set~$R$ of parameters for which the
vector field is analytically conjugate to a constant normal form, giving rise to
a discrete set of eigenfuctions~$\ph_{mn}$ which are analytic functions on the
torus; he then writes:
\emph{It is possible that the dependence of~$\ph_{mn}$ on the parameter~$\th$ on
the set~$R$ is related to the class of functions of the type of monogenic Borel
functions and, despite its everywhere-discontinuous nature, will admit
investigation by appropriate analytical tools.}
} 
We provide evidence that Kolmogorov's intuition was correct by establishing a
monogenic regularity result upon a complexified rotation number for the KAM
curves of a family of analytic twist maps of the annulus; as a consequence of
our result, these curves enjoy a property of ``$\gH^1$-quasianalyticity'' with
respect to the rotation number.


We recall that small divisor problems are at the heart of the study of
quasiperiodic dynamics: resonances are responsible for the possible divergence
of the perturbative series expansions of quasi-periodic motions and their
accumulation must be controlled in order to prove convergence.
The KAM (Kolmogorov-Arnold-Moser) theory deals with perturbations of completely
integrable Hamiltonian systems for which, when frequencies verify a
suitable Diophantine condition, the small divisor difficulty can be overcome and
one can establish the persistence of quasi-periodic solutions of Hamilton's
equations;
in the analytic case their parametric expressions depend analytically on angular
variables as well as on the perturbation parameter, however they are only
defined on closed sets in frequency space, corresponding to the Diophantine
condition. 


Borel's monogenic functions may be considered as a substitute to holomorphic
functions when the natural domain of definition is not open but can be written as
an increasing union of closed subsets of the complex plane
(monogenicity essentially amounts to Whitney differentiability in the complex
sense on these larger and larger closed sets).
According to what these closed sets are, monogenic functions may enjoy some of
the properties of holomorphic functions (\eg one may be able to use the Cauchy
integral).
As pointed out by Herman \cite{He}, Borel's motivation was probably to ensure
quasianalytic properties (unique monogenic continuation) by an appropriate choice
of the sequence of closed sets, which turns out to be difficult in a general
framework.


Kolmogorov's question about the link between small divisor problems and Borel's
monogenic functions has already been considered in small divisor
problems other than KAM theory, particularly in the context of circle maps
\cite{Ar}, \cite{He}, \cite{Ri} where the role of frequency is played by the
so-called rotation number (see also \cite{BMS}, \cite{MS1}, \cite{CM},
\cite{MS2}).
In his work on the local linearization problem of analytic diffeomorphisms of
the circle,
Arnold \cite{Ar} defined a complexified rotation number, \wrt\ which he showed
the monogenicity of the solution of the linearized problem, but his method did
not allow him to prove that the solution of the nonlinear conjugacy problem was
monogenic, because it would have required to iterate infinitely many times a
process in which the analyticity strip was reduced by a finite amount equal to
the imaginary part of the rotation number.
This point was dealt with by Herman \cite{He}, who used quite a different method
and also reformulated Borel's ideas using the modern terminology,
and by Risler \cite{Ri}, who used Yoccoz's renormalization method \cite{Yo88},
\cite{Yo95} to enlarge the set of complex rotation numbers covered by the
regularity result, passing from Siegel's Diophantine condition to Bruno's
condition.

In this article, we consider the Lagrangian formulation of KAM theory for
symplectic twist maps of the annulus \cite{SZ}, \cite{LM} and prove that the
parametrization of the invariant KAM curves can be extended to complex values of
the rotation number, that their dependence on real and complex rotation numbers
is an example of Borel's monogenic function and furthermore that it enjoys the
property introduced in \cite{MS2} under the name ``$\gH^1$-quasianalyticity''. 
This is sufficient to get an interesting uniqueness property for the monogenic
continuation of the KAM curve because such functions are determined by their
restriction to any subset of their domain of definition which has positive
linear measure.


With a view to avoid too many technicalities, we do not try to work in the most
general context. 
We restrict ourselves to the standard family of area-preserving
twist maps of the annulus because we find it suggestive enough and this allows
us to contrast our monogenicity result, which entails holomorphy with respect to
complex non-real frequencies, with another result that we prove, according to
which real frequencies do constitute a natural boundary.

We do not try either to reach optimal results for the arithmetical condition; we
content ourselves with imposing a Siegel-type Diophantine condition on the
rotation number (instead of a Bruno-type condition).


Notice that Whitney smooth dependence on \emph{real} Diophantine frequencies has
been established long ago by Lazutkin \cite{Laz} and P\"oschel \cite{Poschel} in
this kind of small divisor problem, but the question we address in this
article is quite different in its spirit:
what is at stake here is the complex extension, its regularity and the
uniqueness property this regularity implies (see
Section~\ref{seccomplexversusreal} below for a comment).
Indeed, we show that from the point of view of classical analytic continuation,
the real axis in frequency space appears as a natural boundary, because of the
density of the resonances, but our quasianalyticity result is sufficient to
prove that some sort of ``generalized analytic continuation''%
\footnote{We borrow this expression from \cite{RS}.}
through it is indeed possible: the knowledge of the parametrizations on a set of
positive linear measure of rotation numbers (real or complex) is sufficient to
determine all the parametrized KAM curves: in this sense there is only one KAM
curve, parametrized by one monogenic function of the rotation number.


\subsubsection*{Plan of the article}


In Section~\ref{secstatres} we formulate the KAM problem which is investigated
in this article: the existence of analytic invariant curves for the
standard family of area-preserving twist maps of the cylinder
and the dependence of the parametrization of the invariant curves on the
rotation number~$\om$ when~$\om$ is allowed to take real and complex values.
We state our main result, Theorem~\ref{thmtiuCunhol}, about the
$\gC^1$-holomorphy of this parametrization, and connect it with Borel's
monogenic functions and their quasianalytic properties.
We also give a Theorem~\ref{thmrat} about the impossibility of having an
analytic continuation in the Weierstra\ss{} sense through the real frequencies.


In Section~\ref{secbeginpf}, we introduce an algebra norm on the space of
$\gC^1$-holomorphic functions, which is useful to deal with nonlinear analysis (in
particular composition of functions).
In Section~\ref{seccohom}, we provide the small divisor estimates for the
linearized conjugacy equation: these estimates must be uniform \wrt\ both real
Diophantine and complex rotation numbers.


In Section~\ref{secLMscheme}, inspired by Levi-Moser's ``Lagrangian'' proof of
the KAM theorem for twist mappings \cite{LM}, we adapt their algorithm to
construct a sequence of approximations which converges to the parametrization of
the invariant curve in our Banach algebra of $\gC^1$-holomorphic functions, so
as to prove our complexified KAM theorem (Theorem~\ref{thmtiuCunhol}).


Section~\ref{secrat} is devoted to proving that the real line in the complex
frequency space is a natural boundary, in the classical sense, for the analytic
continuation of the parametrization of the invariant curve
(Theorem~\ref{thmrat}).


\section{Statement of the results}	\label{secstatres}


\subsection{\Igs\ for the standard family}	\label{secigs}

Let $f$ be a $1$-periodic real analytic function with zero mean value. 
We consider the standard family,
\ie the discrete dynamical system defined by
\begin{equation}	\label{eqdefstdfam}
T_\eps \col (x,y) \mapsto (x_1,y_1), \qquad
\left\{ \begin{aligned}
x_1 &= x + y + \eps f(x) \\
y_1 &= y + \eps f(x)
\end{aligned} \right.
\end{equation}
in the phase space $\TT\times\RR$, where $\TT = \RR/\ZZ$ and $\eps$ is a real
parameter
(when $f(x)=\cos(2\pi x)$, $T_\eps$ is the so-called standard map).
For $\eps$ close to~$0$, this is an exact symplectic map that we can view
as a perturbation of the integrable twist map $(x,y)\mapsto (x+y,y)$.

We are interested in the KAM curves associated with Diophantine frequencies.
For $\om \in \RR - \QQ$, we call \emph{\ig\ of frequency~$\om$ for~$T_\eps$} the
graph $G=\big\{\big(x,\ph(x)\big)\big\} \subset\TT\times\RR$ of a continuous map
$\ph\col\TT\to\RR$
such that $T_\eps$ leaves~$G$ invariant 
and the map $\Phi\col x\mapsto x+\ph(x)+\eps f(x)$ on~$\RR$ is conjugate to the
translation $x\mapsto x+\om$ by a homeomorphism of~$\RR$ of the form $\id+u$,
where~$u$ is a $1$-periodic function
(observe that~$\Phi$ is a lift of the circle map induced by the restriction
of~$T_\eps$ to~$G$).

There is a natural way of viewing an invariant graph of frequency~$\om$ as a
parametrized curve $G = \ga(\TT)$:
finding~$G$ is equivalent to finding continuous functions $u, \ph \col\TT\to\RR$
such that $\th\mapsto U(\th) \defeq \th+u(\th)$ defines a homeomorphism of~$\RR$
and the curve
$\ga \col \TT \to \TT \times \RR$
defined by $\ga(\th) = \big( U(\th), \ph(U(\th)) \big)$
satisfies
\begin{equation}	\label{eq:invcur} 
\ga(\th+\om) = T_\eps\big(\ga(\th)\big), \qquad \th\in\TT.
\end{equation}
Setting $v = -\om + \ph\circ U$, \ie writing the curve~$\ga$ as
\begin{equation}	\label{eqgaintermsofuv}
\ga(\th) = \big( \th + u(\th), \om + v(\th) \big)
\end{equation}
%
we see that equation~\eqref{eq:invcur} is equivalent to the system
\begin{gather}
\label{eqvintermsofuom}
v(\th) = u(\th) - u(\th-\om) \\
\label{eqseconddiffuom}
u(\th+\om) - 2 u(\th) + u(\th-\om) = \eps f\big( \th + u(\th) \big)
\end{gather}
(using the fact that $x_1 = x+y_1$ in~\eqref{eqdefstdfam} and
writing~\eqref{eqgaintermsofuv} at the points~$\th$ and $\th+\om$).
It is in fact sufficient to know the function~$u$:
any $1$-periodic solution~$u$ of the second-order difference
equation~\eqref{eqseconddiffuom} such that $\id+u$ is injective parametrizes an
\ig\ of frequency~$\om$ through formulas~\eqref{eqgaintermsofuv}--\eqref{eqvintermsofuom}.

The continuity of~$\ph$ and the irrationality of~$\om$ are enough to ensure
uniqueness:
if it exists, the \ig\ of frequency~$\om$ is unique\footnote{%
The argument relies on the positive twist map condition verified by~$T_\eps$:
Suppose that $G$ and~$G^*$ are \igs\ of frequency~$\om$ and
define $\ph,\Phi,\ga$ and~$\ph^*,\Phi^*,\ga^*$ accordingly; 
if $G$ and~$G^*$ did not intersect, we would have $\ph<\ph^*$ or $\ph>\ph^*$ on~$\TT$,
which would imply $\Phi<\Phi^*$ or $\Phi>\Phi^*$ on~$\RR$, and hence contradict
$\om\in\RR-\QQ$ (as it is known for rotation numbers \cite{HK} that $\Phi_1<\Phi_2$ implies
$\rho(\Phi_1)<\rho(\Phi_2)$ or $\rho(\Phi_1)=\rho(\Phi_2)\in\QQ$);
we thus can find $\th_0,c\in\TT$ such that 
$\ga(c+\th_0) = \ga^*(\th_0)$,
but the shifted curve $\th \mapsto \ga(c+\th)$ is a solution
of~\eqref{eq:invcur} as well as~$\ga$ and iterating this
equation we find 
$\ga(c+\th_0+k\om) = T_\eps^k\big(\ga(c+\th_0)\big) = T_\eps^k\big(\ga^*(\th_0)\big) = \ga^*(\th_0+k\om)$ 
for all integers~$k$;
the irrationality of~$\om$ thus implies that $\ga^*$ coincides with the
shifted curve on a dense subset of~$\TT$, hence everywhere by continuity, whence
$G=G^*$.
}
and the corresponding parametrization~$\ga$ is unique up to a shift in the
variable~$\th$.
We can then normalize the parametrization by adding the requirement that~$u$
have zero mean value:
\emph{equation~\eqref{eqseconddiffuom} cannot have more than one continuous solution~$u$
of zero mean value such that $\id+u$ is injective and finding such a solution is
equivalent to finding an \ig\ of frequency~$\om$.}

The classical KAM theorem for twist maps \cite{Moser} guarantees the existence of an \ig\ of
frequency~$\om$ for every Diophantine~$\om$ provided~$|\eps|$ is small enough.
More precisely, 
given $\om\in\RR-\QQ$ for which there exist $M>0$ and $\tau\ge0$ such that
%
%
$ | \om - \frac{n}{m} | \ge \frac{1}{M m^{2+\tau}} $
for all $(n,m)\in\ZZ\times\NN^*$, 
the map $T_\eps$ admits an \ig\ of frequency~$\om$ as soon as $|\eps|$ is smaller
than a constant~$\rho$ which depends only on $f,M,\tau$.
Moreover, the corresponding curve~$\ga_\om$ is known to be analytic in the
angle~$\th$ and to depend analytically on~$\eps$.

The aim of the present article is to investigate the regularity of the map $\om
\mapsto \ga_\om$, which for the moment is defined on the set of real Diophantine
numbers.
More specifically, we are interested in the quasianalytic properties of this
map;
this will lead us to extend it to certain complex values of the frequency~$\om$.

\subsection{$\gC^1$-holomorphy and $\gH^1$-quasianalyticity of a complex extension}

Throughout the article, we use the notation
\begin{equation}	\label{eqdefSRDDrho}
S_R = \{ x\in\CC/\ZZ \mid |\IM x| < R \}, \quad
\DD_\rho = \{\eps\in\CC \mid |\eps|<\rho\}
\end{equation}
for any $R,\rho>0$.
Let $R_0>0$ be such that $f$ extends holomorphically to a neighbourhood of the
closed strip~${\ov S}_{R_0}$.
Let $R\in(0,R_0)$.
We shall be interested in \igs\ whose parametrizations extend
holomorphically for $\th\in S_R$ and $\eps\in\DD_\rho$ for some $\rho>0$.
Let
\begin{equation}	\label{eqdefBBRrho}
\BB_{R,\rho} = H^\infty(S_R\times\DD_\rho)
\end{equation}	
denote the complex Banach space of all bounded holomorphic functions of
$S_R\times\DD_\rho$.

We fix $\tau>0$ and consider
\begin{equation}	\label{ineqomDiophMtau}
A_M^\RR = \left\{\, \om \in \RR \mid
\forall (n,m)\in\ZZ\times\NN^*,\; 
| \om - \frac{n}{m} | \ge \frac{1}{M m^{2+\tau}} \,\right\}
\end{equation}
for $M > M(\tau)=2\ze(1+\tau)$ (Riemann's zeta function): this is a closed
subset of the real line, of positive measure, which has empty interior and is
invariant by the integer translations.
%
%
The KAM theorem gives us a $\rho=\rho(M)>0$ (which depends also on $f$, $R_0$,
$R$ and~$\tau$) and a function 
\[ \om\in A_M^\RR \mapsto u \in \BB_{R,\rho} \] 
such that, for $\eps$ real, the restriction of $u(.\,,\eps)$ to~$\TT$ has zero
mean value and parametrizes
through~\eqref{eqgaintermsofuv}--\eqref{eqvintermsofuom} an \ig\ of
frequency~$\om$.

Clearly, $u$ depends $1$-periodically on~$\om$, we shall thus rather
view it as a function of $q = E(\om)$, where
\[
E \col \om \in \CC \mapsto q = \ee^{2\pi\I\om} \in\CC^*,
\]
and consider that we have a function $u_M \col E(A_M^\RR) \to \BB_{R,\rho}$.
Observe that $E(A_M^\RR)$ is a subset of the unit circle~$\SS$ which avoids all
roots of unity and has Haar measure $\ge 1 - \frac{M(\tau)}{M}$.
If we consider $M_1>M$, then we get a larger set $E(A_{M_1}^\RR)$ and, by the
aforementioned uniqueness property,
the corresponding function~$u_{M_1}$ is an extension of~$u_M$, provided we take
$\rho_1\le\rho$ in the KAM result and we regard the new target space
$\BB_{R,\rho_1}$ as containing $\BB_{R,\rho}$.


\begin{figure}

\begin{center}

\psfrag{AM}{$A_M^\CC$}

\epsfig{file=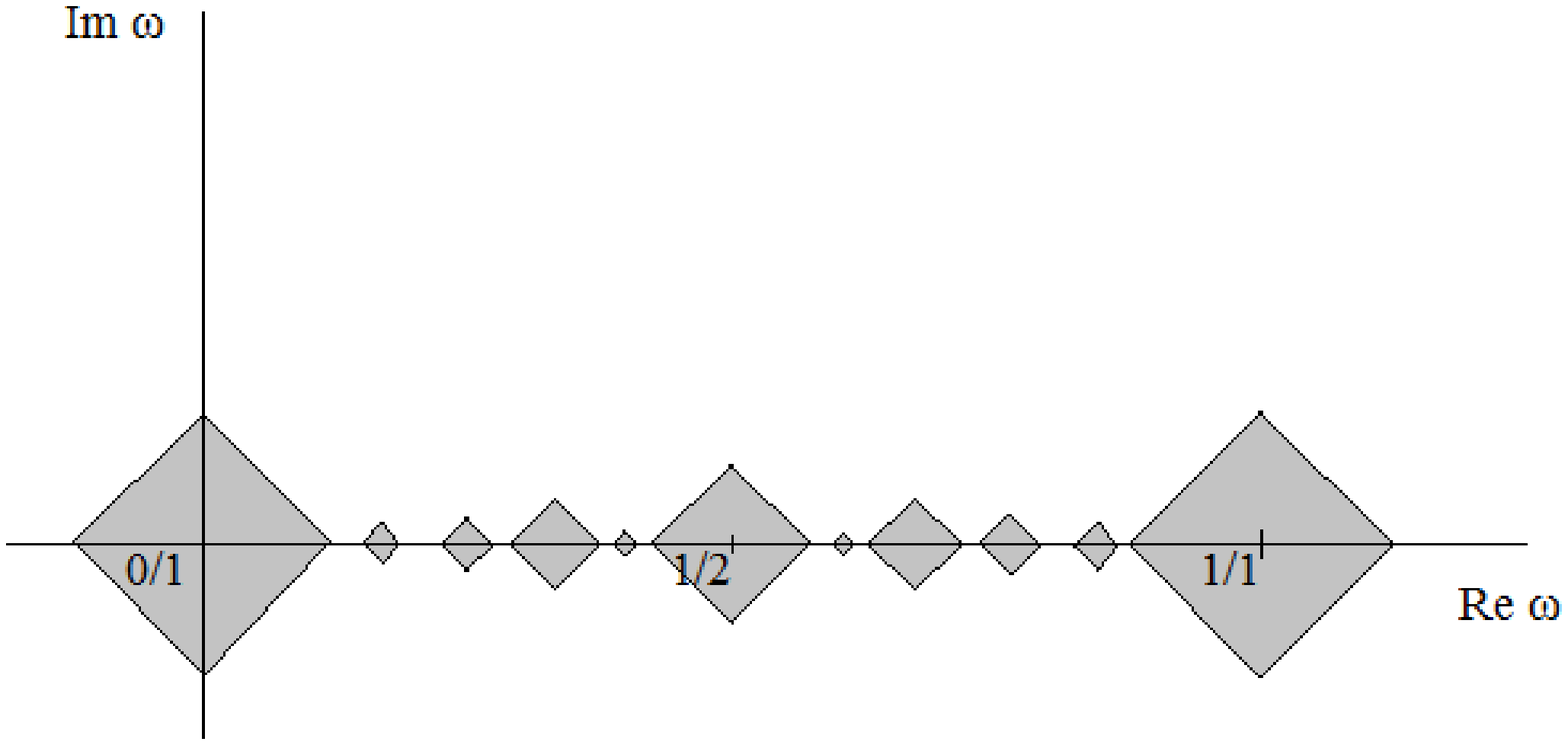,height=1.8in,angle = 0}

\end{center}

\caption{\label{figAMCC} The perfect subset $A_M^\CC\subset\CC$}

\end{figure}



\begin{figure}

\begin{center}

\psfrag{OR}{$0$}

\psfrag{KI}{$K\IN$} \psfrag{KE}{$K\EX$}
\psfrag{GI}{$\Ga\IN$} \psfrag{GE}{$\Ga\EX$}

\epsfig{file=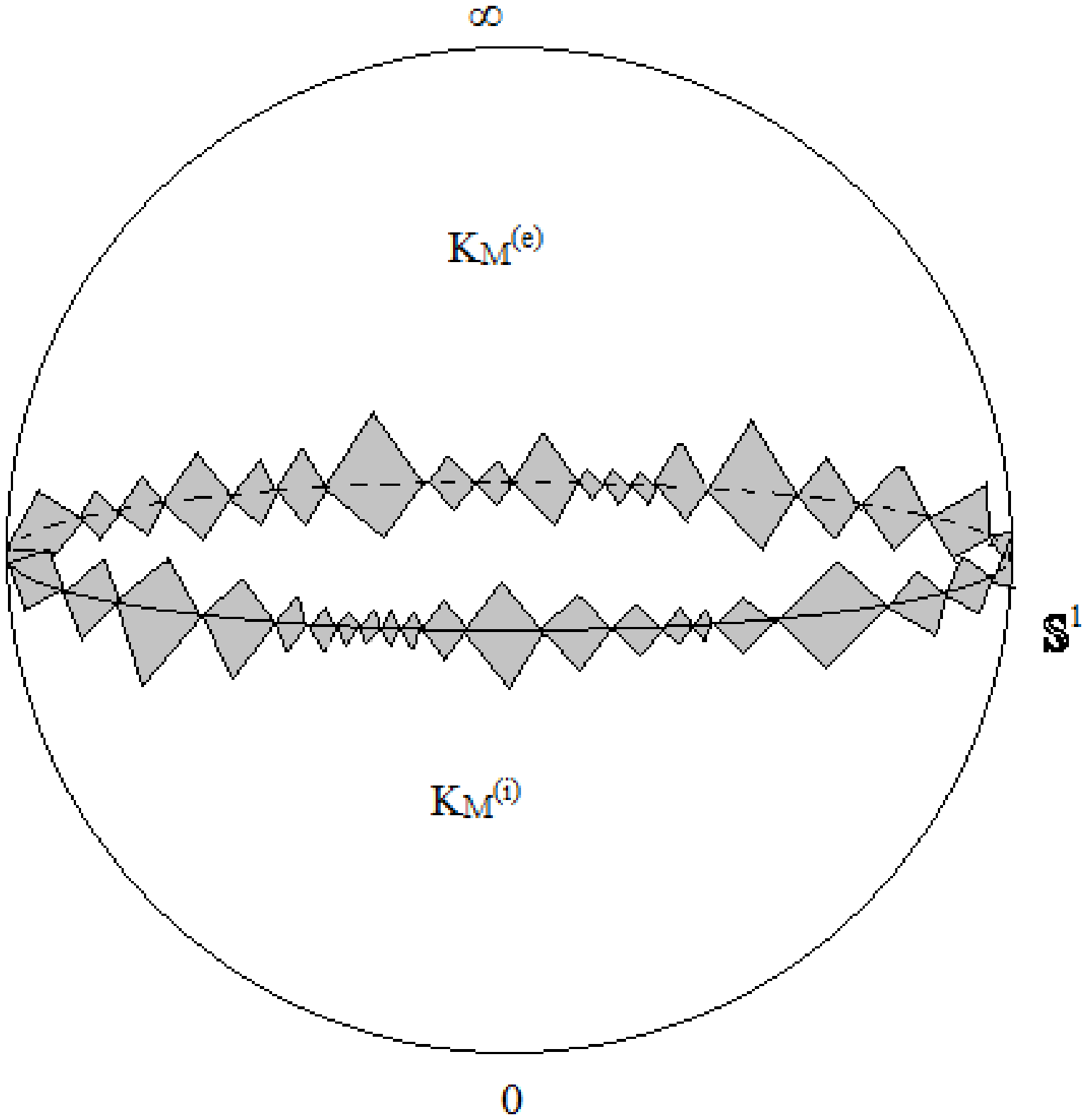,height=3.6in,angle = 0}

\end{center}

\caption{\label{figCurves} The perfect subset $K_M = K_M\IN \cup K_M\EX\subset\PP$}

\end{figure}


Let
\begin{equation}	\label{eqdefAMCC}	
A_M^\CC = \bigl\{\, \om \in\CC \mid\; \exists \om_*\in A_M^\RR
\;\text{ such that }\; |\IM \om| \ge |\om_* - \RE \om| \,\bigr\}
\end{equation}
%
%
and
\begin{equation}	\label{eqdefKM}
K_M = E(A_M^\CC) \cup\{0,\infty\} \subset \PP,
\end{equation}
where $\PP$ denotes the Riemann sphere---see Figures~\ref{figAMCC} and~\ref{figCurves}.
Observe that~$A_M^\CC$ is a perfect subset of~$\CC$ and~$K_M$ is a perfect subset of~$\PP$.

Our main result is that the above function~$u_M$ extends to a $\gC^1$-holomorphic function
from~$K_M$ to~$\BB_{R,\rho}$ (possibly with a smaller~$\rho$).
The reader is referred to Section~\ref{secCunhol} for the definition of the
Banach space $\Chol(K,B)$ of all $\gC^1$-holomorphic functions
from a perfect subset~$K$ of~$\CC$ or~$\PP$ to a Banach space~$B$
($\gC^1$-holomorphy essentially means complex differentiability in the sense of Whitney, \ie
real Whitney differentiability on a closed subset with partial derivatives which satisfy the
Cauchy-Riemann equations).


\begin{theorem}	\label{thmtiuCunhol}
Suppose $\tau>0$, $0<R<R_0$, $f$ real analytic and holomorphic in a
neighbourhood of~${\ov S}_{R_0}$, with zero mean value.
Then there exist $c>0$ (depending on $\tau$, $R_0$, $f$ and~$R$) and,
for each $M>2\ze(1+\tau)$, 
a function 
\[ \ti u_M \in \Chol(K_M, \BB_{R,\rho}),
\quad \text{with 
$\rho = c M^{-8}$}
\]
such that, for each $\om\in A_M^\RR$ and $\eps\in (-\rho,\rho)$, the function
$\th\in\TT \mapsto \ti u_M(\ee^{2\pi\I\om})(\th,\eps)$ has zero mean value and
parametrizes through~\eqref{eqgaintermsofuv}--\eqref{eqvintermsofuom} an \ig\ of
frequency~$\om$ for~$T_\eps$.
\end{theorem}


The proof of Theorem~\ref{thmtiuCunhol} will start in Section~\ref{secbeginpf}.


Theorem~\ref{thmtiuCunhol} provides a function~$\ti u_M$ on~$K_M$ which is an extension
of the function~$u_M$ that we had on~$E(A_M^\RR)$.
\emph{This extension is unique and, if we consider $M_1>M$, then~$\ti u_{M_1}$ is
an extension of~$\ti u_M$} (if we regard $\BB_{R,\rho}$ as a
subspace of~$\BB_{R,\rho_1}$ for $\rho_1<\rho$);
these facts are simple consequences of a quasianalyticity property, which is
established in~\cite{MS2} for all functions $\gC^1$-holomorphic on~$K_M$ and
which we now recall.

Denote by~$\gH^1$ the one-dimensional Hausdorff outer measure associated with
the spherical metric in~$\PP$.
%
%
Suppose that $C$ is a subset of~$\PP$ and that $\lL$ is a linear space of functions,
all of which are defined on~$C$. 
We say that~$\lL$ is \emph{$\gH^1$-quasianalytic relatively to~$C$} if any subset of~$C$ of
positive $\gH^1$-measure is a uniqueness set for~$\lL$ (\ie the only function of~$\lL$
vanishing identically on this subset of~$C$ is $\equiv0$);
in other words, a function of~$\lL$ is determined by its restriction to any subset
of~$C$ of positive $\gH^1$-measure.


Observe that the interior~$\ovi{K}_M$ of the above compact set$K_M$ has two
connected components:
\[
\ovi{K}\str{1.55}\IN_M \defeq \ovi{K}_M \cap
\{\, q\in\PP\mid |q|<1 \,\},
\quad
\ovi{K}\str{1.55}\EX_M \defeq \ovi{K}_M \cap
\{\, q\in\PP\mid |q|>1 \,\}.
\]
It is proved in~\cite{MS2} that~$K_M$ has the property
that, for any Banach space~$B$, the Banach space
\[
\gO(K_M,B) = \{\, \ph \colon K_M \to B \;
\text{continuous in~$K_M$ and holomorphic in~$\ovi{K}_M$}
\,\}
\]
is $\gH^1$-quasianalytic relatively to~$K_M$.
Since $\Chol(K_M,B)\subset\gO(K_M,B)$, this space inherits the
$\gH^1$-quasianalyticity property.%
\footnote{
Thus proving $\ti u \in \gO(K_M,\BB_{R,\rho})$ is sufficient to get the
$\gH^1$-quasianalyticity; this can be achieved with $\rho = c M^{-4}$ by
a simple adaptation of the proof of Theorem~\ref{thmtiuCunhol}.
}
This is why $\ti u_M$ is uniquely determined and must coincide with $\ti u_{M_1
| K_M}$ for $M_1>M$: this function is in fact determined by its restriction~$u_M$
to~$E(A_M^\RR)$, and even by its restriction to any subset of~$E(A_M^\RR)$ of
positive Haar measure of the unit circle.

This quasianalyticity property was our main motivation.
What is striking is that, as we shall see in Section~\ref{seccomm},
finding the complexified function~$\ti u_M$ in restriction to 
$\{\,|\IM\om| \ge h\,\}$,
\ie for $|q|\le \ee^{-2\pi h}$ or $|q|\ge \ee^{2\pi h}$,
is relatively easy because this can be done by solving an equation which does
not involve any small divisor; still, 
the quasianalyticity property shows that
all the real KAM curves determined by~$u_M$
can be obtained from this easy-to-find function by a kind of ``generalized
analytic continuation''
(the restriction of this function to a positive $\gH^1$-measure subset of~$K_M$
is even sufficient).


On the other hand, resonances produce an obstruction to the analytic
continuation in the Weierstra\ss{} sense through any point of the unit circle~$\SS$,
no matter how small $\rho$ is taken:

\begin{theorem}	\label{thmrat}
Suppose $\tau>0$, $0<R<R_0$, $f$ real analytic and holomorphic in a
neighbourhood of~${\ov S}_{R_0}$, with zero mean value but not identically zero.
Let $M>2\ze(\tau+1)$ and $\ti u_M \in \Chol(K_M, \BB_{R,\rho})$ as in
Theorem~\ref{thmtiuCunhol}, possibly with a smaller $\rho>0$.
Consider the restriction of~$\ti u_M$ to~$\ovi{K}\str{1.55}\IN_M$
or to~$\ovi{K}\str{1.55}\EX_M$, 
which is a
$\BB_{R,\rho}$-valued holomorphic function on an open subset of~$\PP-\SS$.
Then, given a point~$q_*$ of the unit circle~$\SS$, this holomorphic function
has no analytic continuation in any neighbourhood of~$q_*$.
%
%
\end{theorem}

The proof of Theorem~\ref{thmrat} is given in Section~\ref{secrat}.


\begin{rem}
For fixed~$M$, one can also fix $\eps \in \DD_\rho$ and consider 
$(q,\th) \mapsto \ti u_M(q)(\th,\eps)$ as an element of
$\Chol\big(K_M,H^\infty(S_R)\big)$, \ie a function of~$\th$ which depends
$\gC^1$-holomorphically on~$q$;
the space $\Chol\big(K_M,H^\infty(S_R)\big)$ enjoys the aforementioned
quasianalyticity property.
\end{rem}

\subsection{Monogenic extension}

Gluing together the $\ti u_M$'s given by Theorem~\ref{thmtiuCunhol}, we get a
\emph{monogenic function in the sense of Borel}.
Here is the precise definition taken from~\cite{MS2} (which is an adaptation of
the definition given in~\cite{He}):
suppose that $(\gK_j)_{j\in\NN}$ is a monotonic non-decreasing sequence of compact
subsets of $\PP$ and $(B_j)_{j\in\NN}$ is a monotonic non-decreasing sequence of
Banach spaces with continuous injections $B_j \hookrightarrow B_{j+1}$;
the corresponding space of monogenic functions is the Fr\'echet space obtained
as the projective limit of Banach spaces
\begin{multline*}
\fM\bigl( (\gK_j), (B_j) \bigr) = \varprojlim \cA_J, \\[1ex]
\cA_J =  \bigcap_{0\le j\le J}\Chol(\gK_j,B_j), \qquad
\norm{\ph}_{\cA_J} = \max_{0\le j\le J} \norm{\ph_{|\gK_j}}_{\Chol(\gK_j,B_j)}.
\end{multline*}
As a set, $\fM\bigl( (\gK_j), (B_j)
\bigr)$ thus consists of all the functions~$\ph$ which are defined in
$\bigcup_{j\in\NN} \gK_j$
and such that, for every $j\in\NN$, the restriction $\ph_{|\gK_j}$ belongs to
$\Chol(\gK_j,B_j)$ .

Let us apply this construction,
under the hypotheses of Theorem~\ref{thmtiuCunhol},
with any increasing sequence of positive numbers $(M_j)_{j\in\NN}$ tending
to~$+\infty$ 
(we suppose $M_j > 2\ze(1+\tau)$ for all~$j$)
and with $B_j = \BB_{R,\rho_j}$, $\rho_j = c M_j^{-8}$.
We use the notation
\[
\gM = \fM\bigl( (K_{M_j}), (\BB_{R,\rho_j}) \bigr), \quad
\gF_\tau  = \bigcup_{j\in\NN} K_{M_j}, \quad
\DC_\tau = \bigcup_{M>0} A_M^\RR
\] 
($\DC_\tau$ is the full-measure set of all
real frequencies Diophantine with exponent $2+\tau$).
Observe that 
\[
\gF_\tau = \{\, q\in\PP\mid |q|<1 \,\}
\cup E(\DC_\tau) \cup
\{\, q\in\PP\mid |q|>1 \,\},
\]
with $\gF_\tau \cap \SS = E(\DC_\tau)$ a subset of full Haar measure of the unit
circle~$\SS$,
and that the elements of~$\gM$ may be viewed as functions from~$\gF_\tau$ to the
space $H^\infty(S_R)\{\eps\}$ (holomorphic germs in~$\eps$ with values in
$H^\infty(S_R)$).
Theorem~\ref{thmtiuCunhol} immediately yields


\begin{corollary}
There is a function $\ti u \in \gM$ such that, for each $\om\in\DC_\tau$ and for
each real $\eps$ close enough to~$0$, the function
$\th\in\TT \mapsto \ti u(\ee^{2\pi\I\om})(\th,\eps)$ has zero mean value and
parametrizes through~\eqref{eqgaintermsofuv}--\eqref{eqvintermsofuom} an \ig\ of
frequency~$\om$ for~$T_\eps$. 
\end{corollary}


As explained in~\cite{MS2}, the space~$\gM$ is $\gH^1$-quasianalytic relatively
to~$\gF_\tau$, so we can say that~$\ti u$ is determined by its restriction to any
positive $\gH^1$-measure subset of~$\gF_\tau$.
In particular, we emphasize that the function $\ti u_{|E(\DC_\tau)}$ which encodes
the real KAM curves is determined by the restriction of~$\ti u$ to any such
subset and we repeat that finding this restriction is easy when the subset is
contained in $\{\,|q|\le \ee^{-2\pi h}\,\}$ or $\{\,|q|\ge \ee^{2\pi h}\,\}$.

\emph{Thus, we have a single analytic object, $\ti u$, which determines all the real KAM
curves, as if there were only one KAM curve instead of separate \igs, each of
which with its own Diophantine frequency $\om\in\DC_\tau$.}


The interior~$\ovi{\gF_\tau}$ of~$\gF_\tau$ has two connected components,
$\{\, q\in\PP\mid |q|<1 \,\}$ and $\{\, q\in\PP\mid |q|>1 \,\}$,
and each function of~$\gM$ is holomorphic in~$\ovi{\gF_\tau}$.
The $\gH^1$-quasianalyticity thus implies a form of coherence:
if two functions of this space coincide on one of the connected components
of~$\ovi{\gF_\tau}$, then they coincide on the whole of~$\gF_\tau$.
Given a function like~$\ti u$, we may think of 
the outside function $\ti u_{|\{\,|q|>1 \,\}}$ 
as of a ``generalized analytic continuation'' of 
the inside function $\ti u_{|\{\,|q|<1 \,\}}$ (see \cite{RS}),
although, according to Theorem~\ref{thmrat}, 
classical analytic continuation in the sense of Weierstra\ss{} is
impossible across the unit circle~$\SS$.
%
%
These two holomorphic functions give rise to a boundary value function $\ti
u_{|E(\DC_\tau)}$ which encodes the real KAM curves.


\begin{rem}
Instead of keeping~$\tau$ fixed as we did in the previous discussion, one can
also let it vary so as to reach the set of all Diophantine real numbers
\[
\DC \defeq \bigcup_{\tau>0} \DC_\tau.
\] 
One just needs to take any sequences $\tau_j \uparrow +\infty$ and $M_j \uparrow
+\infty$.
Observing that
\[
\DC = \bigcup_{j\in\NN} A_{\tau_j,M_j}^\RR
\] 
(with $A_{\tau,M}^\RR$ defined by~\eqref{ineqomDiophMtau} but we now make
explicit the dependence on~$\tau$)
and using correspondingly the space of monogenic functions associated with
$K_{\tau_j,M_j}$ (instead of~$K_{\tau,M_j}$ with fixed~$\tau$) and
$\BB_{R,\rho_j}$ with $\rho_j \defeq c(\tau_j) M_j^{-8}$ (with $c(\tau)$ as in
Theorem~\ref{thmtiuCunhol}), we get a monogenic~$\ti u$ defined on
$\gF \defeq \bigcup K_{\tau_j,M_j}= \{\, q\in\PP\mid |q|<1 \,\}
\cup E(\DC) \cup
\{\, q\in\PP\mid |q|>1 \,\}.$
\end{rem}


\begin{rem}
Stronger than $\gH^1$-quasianalyticity is the following more classical property:
if $\lL$ is a linear space of functions defined in $C\subset\PP$, all of which
admit an asymptotic expansion at a point $q_0\in C$, we say that $\lL$ is
quasianalytic at~$q_0$ in the sense of Hadamard if the only function of~$\lL$ with
zero asymptotic expansion at~$q_0$ is $\equiv0$, \ie a non-tricial function
of~$\lL$ cannot be flat at~$q_0$ (see \cite{MS2}).
A convergent Taylor series is a particular case of asymptotic expansion, thus if
the functions of~$\lL$ are analytic at~$q_0$ the question of the quasianalyticity
at~$q_0$ makes sense; the question is not trivial when the functions are
holomorphic in the interior of~$C$ but this interior is not connected.

This is what happens with our spaces $\gO(K_M,B)$ and~$\gM$:
given any $q_0\in\PP$ with $|q_0|\neq1$, the function~$\ti u$ which
encodes all the KAM curves has a convergent Taylor series at~$q_0$ and this
Taylor series determines~$\ti u$ everywhere.
In the case of $q_0=0$, this Taylor series is particularly easy to compute
inductively (see formula~\eqref{eqinduc0} in Section~\ref{seccomm}).%
\footnote{
Observe that this Taylor series $\sum_{n\ge1} q^n u_n$ also determines the
function~$f$ defining the dynamical system~$T_\eps$ too
(the $k$th Fourier coefficient of~$\eps f$ coincides with the $k$th Fourier
coefficient of $u_{|k|}$), 
which may be considered as a kind of inverse scattering.
}
See the end of Section~\ref{seccomm} for an open question about the Hadamard
quasianalyticity.
\end{rem}

\subsection{Comments}	\label{seccomm}

\subsubsection*{\Ics\ with complex frequencies for the complexified map}

The extension~$\ti u$ will be obtained as a solution of~\eqref{eqseconddiffuom},
viewed as a complexified difference equation involving the holomorphic extension of~$f$ to
the strip~$S_{R_0}$ and a complex frequency $\om\in A_M^\CC$.
This corresponds to determining a complex invariant curve for the holomorphic
map $\ti T_\eps \col S_{R_0} \times \CC \to \CC/\ZZ\times\CC$ defined
by~\eqref{eqdefstdfam}. 
(In fact, the hypothesis that~$f$ be \emph{real} analytic is not necessary: $f$
holomorphic in~$S_{R_0}$ is sufficient.)

This might be surprising at first sight:
if $R<R_0<\infty$, $\th\in S_R$ and $|\IM\om|$ is too large, we cannot prevent
$\th \pm \om$ from lying out of the strip~$S_R$ where the solution~$\ti u$ is
supposed to be defined;
what is then the meaning of the \lhs\ of~\eqref{eqseconddiffuom}?
The explanation is that, in fact,~$\ti u$ will be defined in the larger strip
$S_{R+|\IM\om|}$.

This can be viewed as an effect of the regularizing effect of the operator~$E_q$
which is defined via Fourier series by the formulas
\[
E_q \col \ph = \sum_{k\in\ZZ} \hat\ph_k e_k 
\mapsto \psi = \sum_{k\in\ZZ^*} \hat\psi_k e_k, \qquad
\hat\psi_k \defeq
\frac{1}{\ee^{2\pi\I k\om}-2+\ee^{-2\pi\I k\om}} \hat\ph_k
= \frac{1}{q^k-2+q^{-k}} \hat\ph_k
\]
(with the notation of Section~\ref{secFourier})
and which has the property 
$\psi(\th+\om)-2\psi(\th)+\psi(\th-\om) = \ph(\th)-\mean\ph$.
Indeed, when $|\IM\om|\ge h>0$, 
one can check that $\ph$ holomorphic in $S_R$ implies $\psi=E_q\ph$ holomorphic in
$S_{R+h}$;
on the other hand, equation~\eqref{eqseconddiffuom} with the requirement $\mean
u=0$ is equivalent to
\begin{equation}	\label{eqfixedpt}
u = \eps E_q\big( f \circ (\id+u) \big)
\end{equation}
and the vanishing of $\be = \eps \mean{f \circ (\id+u)}=0$.
For $|q|\neq1$, it is easy to find a solution $u(q)$ of equation~\eqref{eqfixedpt} by means of a
fixed point method and to check that it depends holomorphically on~$q$; the
difference~$\cE(u)$ between the \rhs\ and the \lhs\ of~\eqref{eqseconddiffuom} is then
the constant $\be$ and it is $\equiv0$ because of Lemma~\ref{zeromean}
($\cE(u)=\be$ constant implies $\be = \mean{(1+\pth u)\cE(u)}$).

The Taylor series at $q=0$ of the solution is obtained as follows:
the operator~$E_q$ can be expanded as 
\[
E_q = \sum_{n\ge1} q^n E\suppar n, \qquad
E\suppar n \col \sum_{k\in\ZZ} \hat\ph_k e_k \mapsto
\sum d \left( \hat\ph_m e_m + \hat\ph_{-m} e_{-m} \right),
\]
where the last summation is over all factorizations $n=md$, with integers
$m,d\ge1$;
then we have $u = \sum_{n\ge 1} q^n u_n$ with $u_n\in \BB_{R,\rho_M}$,
convergent for $|q| < \ee^{-2\pi/M}$, with
$u_1 = \eps E\suppar 1 f$ and
\begin{equation}	\label{eqinduc0}
u_n = \eps E\suppar n f + \eps \sum E\suppar{n_0} \big( 
\frac{1}{r!}f\suppar r u_{n_1}\cdots u_{n_r} \big),
\end{equation}
where the last summation is over all decompositions $n = n_0+\cdots+n_r$ with
integers $r,n_0,\ldots,n_r\ge1$.
Observe that each~$u_n$ is a polynomial in~$\eps$ and a trigonometric polynomial
in~$\th$, with Fourier coefficients $\cF_k(u_n)=0$ for $|k|>n$ and 
$\cF_{\pm n}(u_n) = \eps \cF_{\pm n}(f)$.


\subsubsection*{Complex versus real Whitney differentiability}
\label{seccomplexversusreal}

For a closed subset~$A$ of~$\RR^n$, a Banach space~$B$ and a function $\ph\col
A\to B$, 
the definition of being~$\gC^1$ or~$\gC^\infty$ in the sense of Whitney is
intrinsic: it involves only the set~$A$ on which~$\ph$ is defined (and the partial
derivatives of~$\ph$ are uniquely determined as soon as~$A$ is perfect).

Whitney's extension theorem yields an alternative definition
(see \cite{Whitney}, \cite{Stein}):
$\ph$ is Whitney $\gC^1$ on~$A$ if and only if there exists
a function~$\ti\ph$ which is $\gC^1$ on~$\RR^n$ such that $\ti\ph_{|A}=\ph$.
Of course, the extension~$\ti\ph$ is in general not unique.

Similarly, if~$K$ is a perfect subset of $\CC\simeq\RR^2$ (or~$\PP$),
there is an intrinsic definition of being $\gC^1$-holomorphic on~$K$
(see Section~\ref{secCunhol})
and Whitney's extension theorem%
\footnote{
See \cite{ALG}, Remark~III.4 and Proposition~III.8: our $\gC^1$-holomorphic
functions correspond to their W-Taylorian $1$-fields; see also \cite{Glaes},
pp.~65--66.
}
implies that
$\ph$ is $\gC^1$-holomorphic on~$K$ if and only if there exists
a function~$\ti\ph$ which is $\gC^1$ in the real sense on~$\CC$ (or~$\PP$) such
that $\ti\ph_{|K}=\ph$ and $\bar\pa\ph_{|K}=0$.

The classical results of P\"oschel \cite{Poschel} and Lazutkin \cite{Laz} are
concerned with \emph{real} Whitney regularity.%
\footnote{
Moreover they are dealing with the difficult case of finitely differentiable
data rather than analytic ones.
However, \cite{Poschel} (\S~5a, pp. 690--691), prompted by a remark of Arnold,
alludes to a complexified version of his result in the analytic case.
}
Whitney's extension theorem is then useful for instance to estimate the measure
of the union of the KAM tori, but the extension it yields is somewhat arbitrary
and has no direct dynamical meaning, whereas in this article we are interested
in the uniqueness properties that one gets when $\gC^1$-holomorphy is imposed on
a sufficiently large set~$K_M$ or~$\gF_\tau$.
%



\subsubsection*{Possible generalizations and open questions}

The Siegel-type Diophantine condition~\eqref{ineqomDiophMtau} that we use is not
the optimal one for real frequencies: it is shown in \cite{BG} that a Bruno-type
condition is sufficient for the existence of an \ig\ for $|\eps|$ small enough.
Maybe one could work with a set of complex frequencies larger than~$A_M^\CC$,
built from the set of Bruno numbers instead of~$A_M^\RR$, as the one which is
used in \cite{Ri} to prove $\gC^\infty$-holomorphy results in the context of
circle maps.
(Besides, in the case of the Siegel-type condition, we claim no optimality for
the radius $\rho = c M^{-8}$.)

Even sticking to the Siegel Diophantine condition, another issue to be
considered is $\gC^\infty$-holomorphy,
\ie the existence of infinitely many complex derivatives in the sense of Whitney.
One can indeed expect that our function~$\ti u$ belong to a space of functions
$\gC^\infty$-holomorphic on appropriate compact subsets of~$\PP$ containing
$E(\DC_\tau)$; if so, one could then raise the following question, which is
analogous to the open question formulated by Herman at the end of his article on
circle maps \cite{He}:
Given a point~$q_*$ of~$E(DC_\tau)$, is~$\ti u$, or any function of this space of
$\gC^\infty$-holomorphic functions, determined by the sequence of the values of
its derivatives at~$q_*$? In other words, is this space of functions quasianalytic in
the sense of Hadamard at~$q_*$?

Finally, one should consider the adaptation of this circle of ideas to
perturbations of higher-dimensional twist maps or to near-integrable Hamiltonian
systems.




\section{Functional spaces}
\label{secbeginpf}


The adaptation of Levi-Moser's method \cite{LM} which we employ in
Section~\ref{secLMscheme} to reach the $\gC^1$-holomorphy in the frequency
requires the definition of appropriate functional spaces.
The point is to deal as much as possible with Banach algebras (even if this
requires defining norms which differ from the usual ones), so as to ease all the
processes of nonlinear analysis (in particular functional composition).
We thus begin by a new definition.


\subsection{An algebra norm on the Banach space of $\gC^1$-holomorphic functions}
\label{secCunhol}


Let~$B$ be a complex Banach space.

\subsubsection*{Case of a perfect subset of~$\CC$}

Let $K$ be a perfect subset of~$\CC$.
For any two functions $\ph,\psi \col K \to B$, we set 
%
\begin{equation}	\label{eqdefOmphpsi}
\Om_{\ph,\psi} \col K\times K \to B, \qquad
\Om_{\ph,\psi}(q,q') = 
\begin{dcases}
\hfill 0 & \text{if $q=q'$,} \\
\frac{\ph(q')-\ph(q)}{q'-q} - \psi(q) & \text{if $q\neq q'$.}
\end{dcases}
\end{equation}
The usual definition of $\gC^1$-holomorphy can be rephrased as follows:
\begin{quote}
\emph{A function $\ph \col K\to B$ is $\gC^1$-holomorphic iff it is continuous and there exists a
continuous $\psi \col K\to B$ such that $\Om_{\ph,\psi}$ is continuous
on~$K\times K$.}
\end{quote}
Since $K$ has no isolated point, the function~$\psi$ is then unique; we usually
denote it by~$\ph'$ or~$\ph\suppar1$.
Observe that at any interior point of~$K$ the function~$\ph$ is holomorphic
and~$\ph'$ gives the ordinary complex derivative.
Notice also that, if $K$ is contained in an open set~$U$, then the restriction
to~$K$ of any holomorphic function $U\to B$ is $\gC^1$-holomorphic.

We shall impose furthermore that $\ph$, $\ph'$ and~$\Om_{\ph,\ph'}$ be bounded.
With the notation
\[
\de_\ph \col K \to B, \qquad
\de_\ph(q,q') = \ph(q') - \ph(q),
\]
we get a Banach space by setting
\begin{gather}
\label{eqdefCholKB}
\Chol(K,B) \defeq \{\, \ph\col K\to B, \; \text{$\gC^1$-holomorphic, such that
$\nor{\ph}{K,B}<\infty$} \,\}, \\[1ex]
\label{eqdefnorphKB}
\nor{\ph}{K,B} \defeq n_0(\ph) + n_1(\ph) + n_2(\ph), \quad
\left| \begin{aligned}
n_0(\ph) &\defeq \sup_{K} |\ph|, \\ 
n_1(\ph) &\defeq \max\big\{ \sup_{K}{|\ph'|}, \sup_{K\times K}|\de \ph| \big\}, \\ 
n_2(\ph) &\defeq \sup_{K\times K}|\Om_{\ph,\ph'}|.
\end{aligned}\right.
\end{gather}
This norm is not the standard one (see for instance \cite{MS1}) but is
equivalent to it.
An elementary property is
\begin{lemma}	\label{propcrit}
If $\ell: B \to B_0$ is a bounded linear map of norm~$L$ between complex
Banach spaces, then
$\ph \in \Chol(K,B) \ \Rightarrow\  \ell \circ \ph \in  \Chol(K,B_0)$ and
\begin{equation}\label{eq:lincomp}
\nor{\ell\circ\ph}{K,B_0} \le L \nor{\ph}{K,B}.
\end{equation}
Moreover, if $\ell$ is an isometry, then the inequality in~\eqref{eq:lincomp} is in fact an equality.
\end{lemma}


As a matter of fact, $\Chol(K) \defeq \Chol(K,\CC)$ is a Banach algebra and,
if $B$ is a Banach algebra, then $\Chol(K,B)$ is a Banach algebra on the algebra
$\Chol(K)$:

\begin{lemma}	\label{lemBanachAlg}
For any perfect subset~$K$ of~$\CC$,
\begin{enumerate}
\item[i)] $\la \in \Chol(K)$ and $\ph \in \Chol(K,B)$
$\ \Rightarrow\ $
$\la \cdot \ph  \in \Chol(K,B)$ and 
\begin{equation}	\label{ineqnorlaph}
\nor{\la\cdot\ph}{K,B} \le
\nor{\la}{K} \nor{\ph}{K,B};
\end{equation}
moreover, when viewed as an element of $\Chol(K)$, the constant function~$1$ has
norm $\nor{1}{K} = 1$;
\item[ii)] if $B$ is a Banach algebra (not necessarily commutative), then 
$\ph,\psi \in \Chol(K,B)$
$\,\Rightarrow\,$
$\ph \cdot \psi \in \Chol(K,B)$ and
\begin{equation}	\label{ineqnorprod}
\nor{\ph\cdot\psi}{K,B} \le 
\nor{\ph}{K,B} \nor{\psi}{K,B}.
\end{equation}
\end{enumerate}
\end{lemma}

\begin{proof}
The proofs of the two inequalities are similar, we content ourselves with~\eqref{ineqnorprod}.
We set $\chi = \ph\psi$ and $\chi_1 = \ph\psi'+\ph'\psi$; simple computations yield
\begin{gather*}
\de_\chi(q,q') = \ph(q')\cdot\de_\psi(q,q') + \de_\ph(q,q')\cdot\psi(q),\\
\Om_{\chi,\chi_1}(q,q') = \ph(q') \cdot \Om_{\psi,\psi'}(q,q')
+ \de_\ph(q,q') \cdot \psi'(q) + \Om_{\ph,\ph'}(q,q') \cdot \psi(q),
\end{gather*}
whence the $\gC^1$-holomorphy follows, with $\chi' = \chi_1$; moreover, 
\begin{align*}
%
%
n_0(\chi) &\le n_0(\ph) \cdot n_0(\psi), \\
n_1(\chi) &\le n_0(\ph) \cdot n_1(\psi) + n_1(\ph) \cdot n_0(\psi), \\
n_2(\chi) &\le n_0(\ph) \cdot n_2(\psi) +  n_1(\ph) \cdot n_1(\psi) 
+ n_2(\ph) \cdot n_0(\psi), \\
%
%
%
\intertext{whence}
%
%
\nor{\chi}{K,B} 
&\le n_0(\ph) \big( n_0(\psi) + n_1(\psi) + n_2(\psi) \big)
+ n_1(\ph) \big( n_0(\psi) + n_1(\psi) \big) + n_2(\ph) n_0(\psi) \\
&\le \nor{\ph}{K,B} \nor{\psi}{K,B}.
\end{align*}
\end{proof}

\subsubsection*{Case of a perfect subset of~$\PP$}

To be able to apply the $\gH^1$-quasianalyticity result of \cite{MS2}, we needed
to add~$0$ and~$\infty$ to~$E(A_M^\CC)$ in the definition~\eqref{eqdefKM}
of~$K_M$, so as to have a compact subset of~$\PP$ of the appropriate form.
We thus need to explain the definition of $\Chol(K,B)$ when $K$ is a perfect
subset of the Riemann sphere~$\PP$.

Let $\ti K = \{\, \xi\in\CC \mid 1/\xi \in K\,\}$ (with the convention
$1/0=\infty$);
both $\ti K$ and $K\cap\CC$ are perfect subsets of~$\CC$.
Our definition of $\gC^1$-holomorphy is:
\begin{quote}
\emph{A function $\ph \col K\to B$ is $\gC^1$-holomorphic iff 
its restriction $\ph_{|K\cap\CC}$ is $\gC^1$-holomorphic on $K\cap\CC$ 
and the function $\ti\ph \col \xi\in\ti K \mapsto \ph(1/\xi)$ is
$\gC^1$-holomorphic on $\ti K$.}
\end{quote}
We then define
\begin{equation}	\label{defnorKB}
\nor{\ph}{K,B} \defeq \max\big\{ \nor{\ph_{|K\cap\CC}}{K\cap\CC,B}, 
\nor{\ti\ph}{\ti K,B} \big\}
\end{equation}
(\ie we cover~$\PP$ by two charts,
using~$q$ as a complex coordinate in~$\CC$ and
$\xi=\frac{1}{q}$ as a complex coordinate in $\PP - \{0\}$).
Formula~\eqref{defnorKB} defines a norm on the space $\Chol(K,B)$ of all
$B$-valued $\gC^1$-holomorphic functions on~$K$,
which makes it a Banach space.
It is easy to check that Lemma~\ref{propcrit} is still valid, as well as Lemma~\ref{lemBanachAlg}:
\begin{quote}
\emph{$\Chol(K) \defeq \Chol(K,\CC)$ is a Banach algebra and,
if $B$ is a Banach algebra, $\Chol(K,B)$ is a Banach algebra on the algebra
$\Chol(K)$.}
\end{quote}


\subsection{Preliminary small divisor estimates for $q\in K_M$}


Here is an example of a $\gC^1$-holomorphic function which, in a sense,
is the elementary brick of all the solutions of small divisor problems:
\begin{proposition}	\label{lemelbrick}
Let $\tau>0$, $M > 2\ze(1+\tau)$ and~$K_M$ as in~\eqref{eqdefKM}.
Then, for any $k\in\ZZ^*$, the formula 
\begin{equation}	\label{eqdeflak}
\la_k(q) = \frac{1}{q^k-1}
\end{equation}
defines a function
$\la_k \in \Chol(K_M)$ satisfying
\begin{equation}	\label{ineqnorlak}
\nor{\la_k}{K_M} \le 7 M^2 |k|^{3+2\tau}.
\end{equation}
\end{proposition}


Of course, the functions~$\la_k$ are just the restrictions to~$K_M$ of meromorphic functions
but, considering them simultaneously, one easily gets non-meromorphic
$\gC^1$-holomorphic functions;
for instance, for any~$s$ in the open unit disk, the series of functions $\ph =
\sum_{k\ge1} s^k \la_k$ is convergent in $\Chol(K_M)$ and defines a
$\gC^1$-holomorphic function for which the unit circle is a natural boundary
(because of the accumulation of poles at the roots of unity)---this is an example of
Borel-Wolff-Denjoy series (see \eg \cite{MS1}).

The proof of Proposition~\ref{lemelbrick} will make use of use the Diophantine
condition~\eqref{ineqomDiophMtau} in the form of

\begin{lemma}	\label{lemmajorlak}
\begin{equation}	\label{ineqlak}
k\in\ZZ^* \ens\text{and}\ens q\in K_M
\ens\Rightarrow\ens
\left| \frac{1}{q^k-1} \right| \le \sqrt 2 M |k|^{1+\tau}.
\end{equation}
\end{lemma}


\begin{proof}[Lemma~\ref{lemmajorlak} implies Proposition~\ref{lemelbrick}]
We set $\ti K_M = \{\, \xi \in \CC \mid 1/\xi \in K_M \,\}$ and observe that
$\ti K_M = K_M \cap \CC = E(A_M^\CC) \cup\{0\}$
(because $A_M^\CC$ is symmetric \wrt~$0$).
Let $\ti \la_k \col \xi\in\ti K_M \mapsto \la_k(1/\xi)$.
We must check that ${\la_k}_{|\ti K_M}$ and $\ti\la_k \in \Chol(\ti K_M)$
and control their norms.

Since $\la_{k} = \ti \la_{-k}$ and $\ti \la_{k} = \la_{-k}$, we can restrict
ourselves to the case $k\ge1$.

The function $\la \defeq {\la_k}_{|\ti K_M}$, 
being the restriction of a meromorphic function of~$\CC$ with poles off~$\ti K_M$, 
is $\gC^1$-holomorphic on~$\ti K_M$.
Inequality~\eqref{ineqlak} implies
$n_0(\la) \le \sqrt 2 M k^{1+\tau}$
and we compute
\begin{align*}
&\la'(q) = - k \frac{q^{k-1}}{(q^k-1)^2}, \\[1ex]
&\de_{\la}(q,q') = \la(q') - \la(q) = 
-\frac{ q'^k - q^k }{ (q^k-1)(q'^k-1) }, \\[2ex]
&\Om_{\la,\la'}(q,q') = k \frac{q^{k-1}}{(q^k-1)^2} -
\frac{ q'^{k-1} + q'^{k-2}q + \cdots + q' q^{k-2} + q^{k-1} }{%
(q^k-1)(q'^k-1) } \\
& \qquad \qquad \qquad \qquad \qquad = k \frac{q^{k-1}}{(q^k-1)^2} - \sum_{k-1 = a+b}
\frac{q^a}{(q^k-1)} \frac{q'^b}{(q'^k-1)}.
\end{align*}
We have $|\de_\la(q,q')| \le 2 n_0(\la) \le 2\sqrt 2 M k^{1+\tau}$ and, noticing that
\[
0\le \ell \le k \ens\text{and}\ens q\in K_M
\ens\Rightarrow\ens
\left| \frac{q^\ell}{q^k-1} \right| \le \sqrt 2 M |k|^{1+\tau}
\]
(this reinforcement of~\eqref{ineqlak} is obtained by distinguishing the cases
$|q|\le1$ and $|q|\ge1$, and by writing $\frac{q^\ell}{q^k-1} =
-\frac{q^{-(k-\ell)}}{q^{-k}-1}$ in the latter case),
we get
$|\la'(q)| \le 2 M^2 k^{3+2\tau}$
and $|\Om_{\la,\la'}(q,q')| \le 4 M^2 k^{3+2\tau}$.
Since $\sqrt 2 M k^{1+\tau} < M^2 k^{3+2\tau}$ 
(because $\tau>0$ and $M>2\ze(1+\tau)>2$), 
we conclude that
$\nor{\la}{\ti K_M} \le 7 M^2 k^{3+2\tau}$.

On the other hand,
$n_0(\ti\la_k) \le \sqrt 2 M |k|^{1+\tau}$ (still because of~\eqref{ineqlak}) and
$\ti\la_k = -1 -\la$, thus $\ti\la_k' = - \la'$
and $\de_{\ti\la_k} = - \de_{\la}$, $\Om_{\ti\la_k,\ti\la_k'} = -\Om_{\la,\la'}$,
hence $\nor{\ti\la_k}{\ti K_M} \le 7 M^2 k^{3+2\tau}$
and the proof is complete.
\end{proof}


\begin{proof}[Proof of Lemma~\ref{lemmajorlak}]
For $q=0$ or~$\infty$, the result is clearly true. We thus suppose
$q=E(\om)$ with $\om\in A_M^\CC$ and $k\in\ZZ^*$. We must prove that
\begin{equation}	\label{ineqexp}
\left|\ee^{2\pi\I k\om} - 1\right| \ge \frac{1}{\sqrt 2 M |k|^{1+\tau}}.
\end{equation}
This inequality is true if $|\IM(k\om)|\ge\demi$, since one then has 
$\left|\ee^{2\pi\I k\om} - 1\right| \ge \frac{3}{4} > \frac{1}{\sqrt 2}$ 
(as a consequence of
$|\ee^{2\pi\I k\om}| = \ee^{-2\pi\IM(k\om)} >4$ if $\IM(k\om)\le-\demi$
and $<\frac{1}{4}$ if $\IM(k\om)\ge\demi$) and $M>1$.

Inequality~\eqref{ineqexp} is thus a consequence of the existence of $\om_*\in
A_M^\RR$ such that
$|\IM \om| \ge |\om_* - \RE \om|$,
for which $\dist(k\om_*,\ZZ) \ge \frac{1}{M |k|^{1+\tau}}$,
and of the following facts
\begin{gather}
\label{ineqmerom}
\forall z\in \CC, \qquad
|\IM z| \le \demi \ens\Rightarrow\ens
|\ee^{2\pi\I z}-1| \ge \dist(z,\ZZ) \\
\label{ineqrealcomp}
\forall z\in\CC,\; \forall x\in\RR, \qquad
|\IM z| \ge |x-\RE z| \ens\Rightarrow\ens
\dist(z,\ZZ) \ge \frac{1}{\sqrt 2} \dist(x,\ZZ)
\end{gather}
applied to $z=k\om$ and $x=k\om_*$.

\medskip

\noindent Proof of~\eqref{ineqmerom}:
By periodicity, we may suppose $|\RE z|\le\demi$, hence $\dist(z,\ZZ) = |z|$.
It is then sufficient to bound the modulus of
\[
F(z) \defeq \frac{z}{\ee^{2\pi\I z}-1} = \frac{1}{2\pi\I} - \demi z 
+ \frac{1}{\pi\I} \sum_{\ell\ge1} \frac{z^2}{z^2-\ell^2}
\]
(this decomposition results from the identity
$F(z) = \frac{z}{2}(\coth(\pi\I z) - 1)$ and the
classical decomposition of $\coth X$).
But for $|\RE z|$ and $|\IM z|\le\demi$, we have
$|z|^2\le\demi$ and 
\[ 
|z^2-\ell^2|= |\ell-z|\cdot|\ell+z| \ge 
|\ell-\RE z|\cdot|\ell+\RE z| \ge \ell^2 - \frac{1}{4},
\]
whence
$\dst |F(z)| \le \frac{1}{2\pi} + \frac{1}{2\sqrt 2} + 
\frac{1}{2\pi}\sum_{\ell\ge1} \frac{1}{\ell^2 - \frac{1}{4}}
= \frac{3}{2\pi} + \frac{1}{2\sqrt 2} < 1$
in this range.

\medskip

\noindent Proof of~\eqref{ineqrealcomp}:
Let $d>0$ and $p\in \ZZ$ and suppose $|z-p| \le d$; then
\begin{multline*}
\dist(x,\ZZ) \le |x - p| = [\RE z-p + x-\RE z|
\le |\RE(z-p)| + |x-\RE z| \\
\le |\RE(z-p)| + |\IM z|
= |\RE(z-p)| + |\IM(z-p)|
\le \sqrt 2 d.
\end{multline*}
Hence $\dist(x,\ZZ) \le \sqrt 2 \dist(z,\ZZ)$.
\end{proof}


\subsection{Fourier analysis and functional composition in $\Chol\big(K,H^\infty(S_r,B)\big)$}
\label{secFourier}


We shall mainly deal with $\Chol(K_M,B)$ with target spaces $B = \CC$ or
$H^\infty(\DD_\rho)$ or
$\BB_{r,\rho} = H^\infty(S_r\times\DD_\rho)$,
with the notation \eqref{eqdefSRDDrho}--\eqref{eqdefBBRrho}.
Notice that $\BB_{r,\rho}$ is canonically isomorphic to
$\HSrB$ with $B = H^\infty(\DD_\rho)$,
where we denote by $H^\infty(D,B)$ the Banach space of all $B$-valued bounded
holomorphic functions on~$D$ 
endowed with the sup norm
(for any complex manifold~$D$ and any complex
Banach space~$B$) and $H^\infty(D)\defeq H^\infty(D,\CC)$.

Let $B$ denote any complex Banach space. We shall use Fourier analysis in
$\HSrB$, as indicated in the classical

\begin{lemma}	\label{lemFourier}
Let $r>0$ and $k\in \ZZ$.
The formula
\[
\cF_k \col \ph \mapsto \hat\ph_k \defeq 
\int_0^1  \ph(\th) \, \ee^{-2\pi\I k \th}\, \dd\th
\]
defines a bounded linear map $\cF_k \col \HSrB \to B$, with
\[
\norm{\hat\ph_k}_B \le  \ee^{-2\pi r|k|} \norm{\ph}_{\HSrB}.
\]
\end{lemma}

We shall also use the notation
\[
e_k(\th) = \ee^{2\pi\I k\th}, \qquad
\th\in\CC, \; k\in \ZZ
\]
and $\mean\ph \defeq \hat\ph_0$.
Observe that, if $\ph\in\HSrB$, the fact that $|e_k(\th)| = \ee^{-2\pi k\IM\th}$
implies that the series $\sum\hat\ph_k e_k(\th)$ converges to~$\ph(\th)$ for
every $\th\in S_r$, while $\sum \hat\ph_k {e_k}_{|S_{r'}} = \ph_{|S_{r'}}$ is
absolutely convergent in $\HSrpB$ in general only for $r'<r$.

The Cauchy inequalities also yield

\begin{lemma}
If $0 < r' < r$, then the derivation \wrt~$\th$ induces a bounded linear map
$\pa_\th \col \HSrB \to \HSrpB$, with
\[
\norm{\pa_\th^p\ph}_{\HSrpB} \le \frac{p!}{(r-r')^p} \norm{\ph}_{\HSrB},
\qquad p\in\NN.
\]
\end{lemma}

Applying Lemma~\ref{propcrit} with $\ell= \cF_k$ or $\pa_\th^p$, we get

\begin{corollary}
Let $K$ be a perfect subset of~$\CC$ or~$\PP$, $B$ a complex Banach space, $r>0$
and $\ph\in\Chol(K,\HSrB)$. Then:
\begin{enumerate}
\item For all $k\in \ZZ$, the Fourier coefficients 
$\hat\ph_k \defeq \cF_k \circ \ph$ belong to $\Chol(K,B)$ and satisfy
\begin{equation}\label{eq:c1holfou}
\nor{\hat\ph_k}{K,B} \le  \ee^{-2\pi r|k|} \nor{\ph}{K,\HSrB}.
\end{equation}
\item If $0<r'<r$, then $\pa_\th^p\ph$ belongs to $\Chol(K,\HSrpB)$ and satisfies
\begin{equation}\label{eq:c1holder}
\nor{\pa_\th^p\ph}{K,\HSrpB} \le \frac{p!}{(r-r')^p} \nor{\ph}{K,\HSrB}
\end{equation}
for all $p\in\NN$.
\end{enumerate}
\end{corollary}


We now consider composition \wrt\ the variable~$\th$:

\begin{lemma}	\label{lemcompos}
Let $K$ be a perfect subset of~$\CC$ or~$\PP$, $B$ a complex Banach algebra,
$r>r'>0$ and $\ka \in (0,1)$. 
Then, for any $\ph\in\Chol(K,\HSrB)$ and $\psi\in\Chol(K,\HSrpB)$ such that
\[ \nor{\psi}{K,\HSrp} \le \ka(r-r'), \]
the series 
\[
\ph\circ(\id+\psi) = \sum_{p\ge0} \frac{1}{p!} (\pa_\th^p\ph) \psi^p 
\]
is absolutely convergent in $\Chol(K,\HSrpB)$ and 
defines a function which satisfies
\[ \nor{\ph\circ(\id+\psi)}{K,\HSrpB} \le (1-\ka)\ii \nor{\ph}{K,\HSrB}. \]
\end{lemma}

\begin{proof}
Use the Cauchy inequalities~\eqref{eq:c1holder}
and the product inequalities~\eqref{ineqnorprod}.
\end{proof}

We shall use Lemma~\ref{lemcompos} with $B=H^\infty(\DD_\rho)$, 
in which case
\[
\big( \ph\circ(\id+\psi) \big)(q)(\th,\eps) = 
\ph(q)\big(\th+\psi(q)(\th,\eps),\eps\big)
\]
if we use the identification $\Chol(K,\HSrB) = \Chol\big(K,H^\infty(S_r\times\DD_\rho)\big)$.


\section{$\gC^1$-holomorphy of the solution of the cohomological equation}

\label{seccohom}


For $\om\in\RR-\QQ$ and $\ph$ analytic on~$\TT$ of zero mean value, 
the ``cohomological equation'' is the linear equation
\begin{equation}	\label{eqlinnom}
\psi(\th+\om) - \psi(\th) = \ph(\th)
\end{equation}
which will appear in Levi-Moser's scheme in Section~\ref{secLMscheme}.
Its solution is formally given by the Fourier series
\begin{equation}	\label{eqsollinnom}
\psi = \sum_{k\in\ZZ^*} \frac{1}{\ee^{2\pi\I k\om}-1} \hat\ph_k e_k
= \sum_{k\in\ZZ^*} \la_k(q) \hat\ph_k e_k,
\qquad q = \ee^{2\pi\I\om}.
\end{equation}
As is well known, this defines an analytic function~$\psi$ when $\om\in
A_M^\RR$
(use \eg \eqref{ineqlak} with $q$ on the unit circle).
We shall see that this is still true when one considers the complexified
equation associated with $q\in K_M$, 
\ie $\om\in A_M^\CC$ or $\om=\pm\I\infty$, 
and that the solution~$\psi$ depends $\gC^1$-holomorphically on~$q$ when~$\ph$
does.


Notice however that, if $\IM\om\neq0$, equation~\eqref{eqlinnom} requires that
the unknown~$\psi$ be defined in a strip which is larger than the strip
where~$\ph$ admits a holomorphic extension.
Suppose for instance $\IM\om>0$ and $\ph$ holomorphic in~$S_r$; then we seek a
solution~$\psi$ holomorphic for all $\th\in\CC/\ZZ$ such that
$-r < \IM\th < r+\IM\om$,
so that equation~\eqref{eqlinnom} makes sense for all $\th\in S_r$;
it turns out that~\eqref{eqsollinnom} defines such a function.
This will be part of Proposition~\ref{proplinlemma}; before stating it, we
introduce a notation for the functional spaces we shall systematically deal with
from now on.


\begin{notation}
We fix $\tau>0$, $M>2\ze(1+\tau)$ and $\rho>0$ and define~$K_M$ as in~\eqref{eqdefKM}
and 
\begin{equation}	\label{eqdefB}
\Br \defeq H^\infty(\DD_\rho). 
\end{equation}
For any $r>0$, as in \eqref{eqdefSRDDrho}--\eqref{eqdefBBRrho} we define
$\BB_{r,\rho} = H^\infty(S_r\times\DD_\rho)$, which is thus is canonically
isomorphic to $\HSrBr$, and we set
\begin{equation}	\label{eqdefCR}
\Cn r \defeq \Chol(K_M,\BB_{r,\rho}) = \Chol\big( K_M, \HSrBr \big).
\end{equation}
For any $\ph\in \Cn r$, we also use the notation
$\norm{\ph}_r \defeq \nor{\ph}{K_M,\BB_{r,\rho}}$ and
\[
\ph(q,\th) \defeq \ph(q)(\th) \in \Br,
\qquad q\in K_M, \; \th\in S_r,
\]
as well as $\ph(q)(\th,\cdot)$ or $\ph(q,\th,\cdot)$.
\end{notation}


\begin{definition}	\label{defCRpCRm}
For $r>0$, we define~$\Cp r$, \resp $\Cm r$, to be the space of all functions
$\ph = \sum_{k\in \ZZ} \hat{\ph}_k e_k \in \Cn r$ such that each function
\begin{align*}
& \hspace{3em} q^{k} \hat{\ph}_k(q),
\hspace{-3em} & &\text{\resp} & \hspace{-4em}
& \hspace{3em} q^{-k} \hat{\ph}_k(q) \\
\intertext{has limits for $q\to0$ and $q\to\infty$, and the series}
\ph^+(q,\th) &= \sum_{k\in \ZZ} q^{k} \hat{\ph}_k(q) e_k(\th),
\hspace{-3em} & &\text{\resp} & \hspace{-4em}
\ph^-(q,\th) &= \sum_{k\in \ZZ} q^{-k} \hat{\ph}_k(q) e_k(\th), \\
\intertext{%
converges in~$\Br$ for all $(q,\th)\in K_M\times S_r$ and defines an element
of~$\Cn r$.
For any such function, we define}
\nop{\ph}{r} &\defeq \max \big( \norm{\ph}_r, \norm{\ph^+}_r \big),
\hspace{-3em} & &\text{\resp} & \hspace{-4em}
\nom{\ph}{r} &\defeq \max \big( \norm{\ph}_r, \norm{\ph^-}_r \big),\\[1ex]
\na\ph &\defeq \ph^+ - \ph \in \Cn r,
\hspace{-3em} & &\text{\resp} & \hspace{-4em}
\na^-\ph &\defeq \ph - \ph^- \in \Cn r.
\end{align*}
We also set
\[
\Cpm{r} \defeq \Cp r\cap \Cm r , \quad
\nopm{\ph}{r} \defeq \max \big( \nop{\ph}{r}, \nom{\ph}{r} \big).
\]
\end{definition}


It is easy to check that $\big( \Cp r, \nop{\cdot}{r} \big)$, 
$\big( \Cm r, \nom{\cdot}{r} \big)$ and
$\big( \Cpm{r}, \nopm{\cdot}{r} \big)$ are Banach spaces.
Notice that, for $q=\ee^{2\pi\I\om}\in K_M - \{0,\infty\}$,
\begin{equation}	\label{eqphpphm}
\ph^+(q,\th) = \ph(q,\th+\om), \qquad \qquad \ph^-(q,\th) = \ph(q,\th-\om)
\end{equation}
and, if $\IM\om>0$ [\resp if $\IM\om<0$],
\begin{align*}
\ph\in \Cp r &\quad\Rightarrow\quad \ph(q) \in
H^\infty(S_{-r,r+\IM\om},\Br) \quad \big[ \text{\resp}\; H^\infty(S_{-r+\IM\om,r},\Br) \big]\\
\ph\in \Cm r &\quad\Rightarrow\quad \ph(q) \in
H^\infty(S_{-r-\IM\om,r},\Br) \quad \big[ \text{\resp}\; H^\infty(S_{-r,r-\IM\om},\Br) \big]
\end{align*}
with the notation
$S_{h_1,h_2}= \{\, \th \in \CC \mid h_1<\IM (\th) <h_2 \,\}$.

Thus, the cohomological equation~\eqref{eqlinnom} can be rephrased as
$\na\psi=\ph$, where $\ph$ is given in~$\Cn r$ and~$\psi$ is sought in~$\Cp r$.
We shall also need to deal with the shifted equation
$\na^-\psi=\ph$ for which~$\psi$ is sought in~$\Cm r$.
But this is asking too much: we'll have to content ourselves with $\psi\in \Cp{r'}$, \resp
$\Cm{r'}$, with any $r'<r$.


\begin{proposition}	\label{proplinlemma}
Suppose that $0 < r' < r$ and 
$\ph = \sum_{k\in\ZZ} \hat\ph_k e_k \in \Cn r$.
Then the formulas
\begin{equation}	\label{eqdefGaph}
(\Ga\ph)(q) = \sum_{k\in\ZZ^*} \la_k(q) \hat\ph_k(q) e_k, \qquad
(\Ga^-\ph)(q) = - \sum_{k\in\ZZ^*} \la_{-k}(q) \hat\ph_k(q) e_k
\end{equation}
(still with $\la_k(q) = \frac{1}{q^k -1}$)
define two functions $\Ga\ph \in \Cp{r'}$ and $\Ga^-\ph \in \Cm{r'}$, which
satisfy
\begin{gather}
\label{eqGaphsolcoh}
\na(\Ga\ph) = \na^-(\Ga^-\ph) = \ph - \hat\ph_0 \quad\text{in $\Cn{r'}$}. \\[-2.4ex]
\intertext{Moreover,}
\label{eqGaphpGaphm}
(\Ga \ph)^+ = \Ga^- \ph, \quad
(\Ga^- \ph)^- = \Ga\ph \\[-.5ex]
\intertext{%
and there exists a positive constant~$C_1$ which depends only
on~$\tau$ such that, if $r-r' \le 1$, then
}
\label{ineqGaph}
\nop{\Ga\ph}{r'},\, \nom{\Ga^-\ph}{r'} \le 
\frac{C_1 M^2}{(r-r')^{\sig}} \norm{\ph}_r, \\
\label{ineqpaGaph}
\nop{\pa_\th\Ga\ph}{r'},\, \nom{\pa_\th\Ga^-\ph}{r'} \le 
\frac{C_1 M^2}{(r-r')^{\sig+1}} \norm{\ph}_r,
\end{gather}
where  $\sig\defeq 4 + 2\tau$.
\end{proposition}


\begin{proof}
For $k\in\ZZ^*$,
by virtue of~\eqref{ineqnorlak} and~\eqref{eq:c1holfou}, the fact that 
$\norm{e_k}_{H^\infty(S_{r'})} = \ee^{2\pi|k| r'}$ and
$\norm{\pth e_k}_{H^\infty(S_{r'})} = 2\pi |k| \ee^{2\pi|k| r'}$
implies
\begin{gather*}
\norm{\la_k \hat\ph_k e_k}_{r'}, \;
\norm{\la_{-k} \hat\ph_k e_k}_{r'} 
\le 7 M^2 |k|^{3+2\tau} \ee^{-2\pi(r-r')|k|} \norm{\ph}_r, \\
\norm{\la_k \hat\ph_k \pth e_k}_{r'}, \;
\norm{\la_{-k} \hat\ph_k \pth e_k}_{r'} 
\le 7 M^2 (2\pi) |k|^{4+2\tau} \ee^{-2\pi(r-r')|k|} \norm{\ph}_r.
\end{gather*}
The series in~\eqref{eqdefGaph} can thus be viewed as an absolutely
convergent series in~$\Cn{r'}$, defining functions $\Ga\ph,\Ga^-\ph$ which satisfy
\begin{gather*}
\norm{\Ga\ph}_{r'}, \; \norm{\Ga^-\ph}_{r'} 
\le 14 M^2 \Sig(\al,3+2\tau) \norm{\ph}_r, \\
\norm{\pth\Ga\ph}_{r'}, \; \norm{\pth\Ga^-\ph}_{r'} 
\le 14 M^2 (2\pi) \Sig(\al,4+2\tau) \norm{\ph}_r, 
\\[-1ex]
\intertext{where}
\Sig(\al,\be) = \sum_{k\ge1} k^\be \ee^{-\al k}, \qquad
\al = 2\pi(r-r'), \qquad \be > 1.
\end{gather*}
If $r-r'\le1$, then $\al\le 2\pi$.
Since $x\mapsto x^\be \ee^{-\al x}$ is increasing on $(0,\frac{\be}{\al})$,
decreasing on $(\frac{\be}{\al},+\infty)$ and bounded by 
$\left( \frac{\be}{\ee} \right)^\be \al^{-\be}$, 
with $\int_0^\infty x^\be \ee^{-\al x} \,\dd x = \be! \al^{-\be-1}$,
we have
\begin{gather*}
\Sig(\al,\be) 
< \be! \al^{-\be-1} + 
\left( \frac{\be}{\ee} \right)^\be \al^{-\be}
\le \left[ \be! + 2\pi\left( \frac{\be}{\ee} \right)^\be \right] 
\big(2\pi(r-r')\big)^{-\be-1}, \\[-2.4ex]
\intertext{hence}
\norm{\Ga\ph}_{r'} , \; \norm{\Ga^-\ph}_{r'} 
\le \frac{C_1 M^2}{(r-r')^\sig} \norm{\ph}_r, \qquad
\norm{\pth\Ga\ph}_{r'} , \; \norm{\pth\Ga^-\ph}_{r'} 
\le \frac{C_1 M^2}{(r-r')^{\sig+1}} \norm{\ph}_r,
\end{gather*}
with $\sig = 4 + 2\tau$ and
$C_1 \defeq 14 \left[ \sig! + 2\pi\left( \frac{\sig}{\ee} \right)^\sig \right] 
(2\pi)^{-\sig}$.

We now observe that, for each $k\in\ZZ^*$, ${q^k}\la_k(q) = \frac{q^k}{q^k-1} =
-\la_{-k}(q)$ defines an element of $\Chol(K_M,\Br)$, thus the Fourier series
$(\Ga\ph)^+(q) = \sum {q^k}\la_k(q) \hat\ph_k(q) e_k$
coincides with $(\Ga^-\ph)(q)$
and defines an element of~$\Cn{r'}$, hence $\Ga\ph \in \Cp{r'}$ with
$\nop{\Ga\ph}{r'} \le {C_1 M^2}{(r-r')^{-\sig}} \norm{\ph}_r$
and $\nop{\pth\Ga\ph}{r'} \le {C_1 M^2}{(r-r')^{-\sig-1}} \norm{\ph}_r$.
Similarly, $(\Ga^-\ph)^- = \Ga\ph$ (hence~\eqref{eqGaphpGaphm}),
$\Ga^-\ph \in \Cm{r'}$ with
$\nom{\Ga^-\ph}{r'} \le {C_1 M^2}{(r-r')^{-\sig}} \norm{\ph}_r$,
$\nom{\pth\Ga^-\ph}{r'} \le {C_1 M^2}{(r-r')^{-\sig-1}} \norm{\ph}_r$.

The identities~\eqref{eqGaphsolcoh} stem from the relations 
$(q^k-1)\la_k(q) = -(1-q^{-k})\la_{-k}(q) = 1$ valid for all $k\in\ZZ^*$.

%
\end{proof}


\begin{rem}	\label{remker}
Observe that the kernels of~$\na$ or~$\na^-$ do not depend on~$r$ and consist of
the functions which are constant in~$\th$, \ie they coincide with
\begin{equation*}
\Cr \defeq \Chol(K_M,\Br).
\end{equation*}
\end{rem}


\begin{rem}	\label{remGaGamphm}
Observe that
\[
\ph \in \Cm r \ens\Rightarrow\ens
\Ga\ph = \Ga^-(\ph^-) \in \Cpm{r'}
\]
(because $\la_k(q) = - \la_{-k}(q) q^{-k}$ for each $k\in\ZZ^*$),
with
\begin{equation*}
\nopm{\Ga\ph}{r'} \le 
\frac{C_1 M^2}{(r-r')^{\sig}} \nom{\ph}{r}, \qquad
\nopm{\pa_\th\Ga\ph}{r'} \le
\frac{C_1 M^2}{(r-r')^{\sig+1}} \nom{\ph}{r}.
\end{equation*}
\end{rem}


Lastly, as we shall have to keep track of real analyticity \wrt\ the variables
$\om,\th,\eps$, we introduce
\begin{definition}
We denote by $\Cnr{r}$ the subspace of~$\Cn{r}$ consisting of all
functions~$\ph$ whose Fourier coefficients satisfy
\[
\conj\big( \hat\ph_k(q) \big) = \hat\ph_{-k}(1/\bar q),
\qquad k\in\ZZ,\; q\in K_M,
\]
where $\conj$ denotes complex conjugacy in~$\Br$
(\ie $(\conj\psi)(\eps) \defeq \ov{\psi(\bar\eps)}$ for any $\psi\in\Br$).
We also set 
\[
\Cpr r \defeq \Cp r \cap \Cnr r, \quad
\Cmr r \defeq \Cm r \cap \Cnr r, \quad
\Cpmr r \defeq \Cpm r \cap \Cnr r
\]
and $\Crr \defeq \{\, \ph\in\Cr \mid 
\conj\big(\ph(q)\big) = \ph(1/\bar q) \;\text{for all $q\in K_M$}\,\}.$
\end{definition}

The functions in these spaces have the property
\[
|q|=1,\; \th\in\RR/\ZZ,\; \eps\in\RR \ens\Rightarrow\ens \ph(q)(\th,\eps)\in\RR.
\]
Since $\conj\big( \la_k(q) \big) = \la_{-k}(1/\bar q)$, we see that the
operators~$\Ga$ and~$\Ga^-$ preserve real analyticity, so that they induce
operators
\[
\Ga \col \Cnr r \to \Cpr{r'}, \quad
\Ga^- \col \Cnr r \to \Cmr{r'}.
\]


\section{Levi-Moser's modified Newton scheme}	\label{secLMscheme}


\subsection{Reduction of the problem}	\label{sec:reduc}


Suppose $f \in H^\infty(S_{R_0})$ real analytic with zero mean value,
with $f''$ bounded in~$S_{R_0}$ and $0<R<R_0$.
We shall see that this is enough to prove Theorem~\ref{thmtiuCunhol}.

In view of Definition~\ref{defCRpCRm} and formulas~\eqref{eqphpphm}, we can
define the operator
\[
\De = \na - \na^- = \na \na^- = \na^- \na, \qquad
\De \col \Cpm{r} \to \Cn r, \qquad r>0,
\]
so that, for $q=\ee^{2\pi\I\om}\in K_M - \{0,\infty\}$ and $\th\in S_r$,
\begin{equation*}	
(\De u)(q,\th) = u(q,\th+\om) - 2 u(q,\th) + u(q,\th-\om) \in \Br
\end{equation*}
appears as a complexification of the \lhs\ of equation~\eqref{eqseconddiffuom}.
Therefore, the equation
\begin{equation}	\label{eqDeufcompos}
\De u = \eps f \circ (\id+u)
\end{equation}
boils down to equation~\eqref{eqseconddiffuom} when $q=\ee^{2\pi\I\om}$ is on the unit circle and
$\th$ and~$\eps$ are real.


\begin{proposition}
The \rhs\ of~\eqref{eqDeufcompos} makes sense for any $r,\rho>0$ and $u\in \Cpm{r}$
such that $\norm{u}_r < R_0-r$.

Let $r'\in(0,r)$.
If a function $u\in\Cpmr r$ is a solution of~\eqref{eqDeufcompos} which
satisfies 
$\norm{u}_r < \min(R_0-r,r)$ and $\norm{u}_{r'} < r-r'$,
then the formula
\[ \ti u_M(q,\th) \defeq u\big(q,\th-\hat u_0(q)\big) - \hat u_0(q) \]
defines a function $\ti u_M \in \gC^1(K_M,\BB_{r',\rho})$
such that, for each $\om\in A_M^\RR$ and $\eps\in (-\rho,\rho)$, the function
$\th\in\TT \mapsto \ti u_M(\ee^{2\pi\I\om},\th,\eps)$ has zero mean value and
parametrizes through~\eqref{eqgaintermsofuv}--\eqref{eqvintermsofuom} an \ig\ of
frequency~$\om$ for~$T_\eps$.
\end{proposition}

\begin{proof}
The \rhs\ of~\eqref{eqDeufcompos} makes sense because the hypothesis $\norm{u}_r <
R_0-r$ allows us to interpret it according to Lemma~\ref{lemcompos},
viewing $\eps f$ as an element of $\BB_{R_0,\rho} = H^\infty(S_{R_0},\BB_\rho)$,
or even as a function of $\Cn{R_0} = \Chol(K_M,\BB_{R_0,\rho})$ which is
constant in~$q$.

The assumption $\norm{u}_r < r$ and the Cauchy inequalities imply $|\pth
u(q,\th,\eps)| < 1$ for real~$\th$, hence $\th \mapsto U(\th) = \th +
u(q,\th,\eps)$ defines a homeomorphism of~$\RR$ for every $q=\ee^{2\pi\I\om}$ on
the unit circle (\ie $\om\in A_M^\RR$) and $\eps\in(-\rho,\rho)$; according to
Section~\ref{secigs}, this yields an \ig\ for every $\om\in A_M^\RR$, we just
need to shift the parametrization and consider $U(\th-\hat u_0) = \th + \ti
u_M(q,\th,\eps)$ to get the zero mean value normalization.

The fact that $\ti u_M \in \gC^1(K_M,\BB_{r',\rho})$ follows from
Lemma~\ref{lemcompos} if we view $\hat u_0$ as a function of $\Cn{r'}$ (constant
in~$\th$) satisfying 
$\norm{\hat u_0}_{r'} \le \norm{u}_{r'} < r-r'$ and observe that
$\ti u_M = u\circ(\id - \hat u_0) - \hat u_0$.
\end{proof}


\emph{Therefore, to prove Theorem~\ref{thmtiuCunhol}, we only need
to find $c>0$ independent of~$M$ and
a solution $u \in\Cpmr{R_\infty}$ of~\eqref{eqDeufcompos} with $\rho = c M^{-8}$
and some $R_\infty\in(R,R_0)$, such that
\begin{equation}
\norm{u}_{R_\infty} < \min(R_0-R_\infty,R_\infty-R)
\end{equation}
}
(applying the previous proposition with $r=R_\infty$ and $r'=R$).

To obtain this solution~$u$, we shall inductively construct a sequence
$(u_n)_{n\ge1}$ with
\[
u_n \in \Cpmr{R_n}, \qquad 
R_0 > R_1 > R_2 > \ldots, \qquad
R_n \xrightarrow[n\to\infty]{} R_\infty > R
\]
in such a way that the restrictions 
${u_n}_{| K_M \times S_{R_\infty} \times \DD_\rho}$
converge to the desired solution in $\Cpmr{R_\infty}$,
at least if $\rho$ is small enough.
We shall see that the constant~$c$ will depend on~$f$ only
through~$\norm{f}_{R_0}$ and~$\norm{f''}_{R_0}$.

As in \cite{LM}, the passage from~$u_n$ to~$u_{n+1}$ will be a variant of the
Newton method, which we now explain.


\subsection{The inductive step}


Let us define an ``error functional'' as
\begin{equation}	\label{eqdefcE}
u \mapsto \cE(u) \defeq -\De u + \eps f \circ (\id +u),
\end{equation}
so that equation~\eqref{eqDeufcompos} amounts to $\cE(u)=0$.
As noticed earlier, 
\begin{equation}	\label{propcE}
u\in \Cpmr r, \; \norm{u}_r \le R_0-r
\ens\Rightarrow\ens \cE(u) \in \Cnr r.
\end{equation}
The Taylor formula yields 
\begin{align*}
\cE(u+h) &= \cE(u) + \cE'(u)[h] + Q(u,h) \\
\intertext{with a map}
\cE'(u)[h] &\defeq - \De h + \left( \eps f'\circ(\id+u) \right) h \\
\intertext{which is linear in~$h$ and a remainder term}
Q(u,h) &\defeq \left( \int_0^1 \eps f''\circ(\id+u+th) (1-t)\,\dd t \right) h^2
\end{align*}
which has a norm of the same order of magnitude as~$\norm{h}^2$.

The classical Newton method would consist in defining $u_{n+1} = u_n + h$
with~$h$ chosen so that $\cE(u_n) + \cE'(u_n)[h]=0$ and $\norm{h}$ comparable to
$\norm{\cE(u_n)}$, hence a new error
$\cE(u_n+h) = Q(u_n,h)$ which would be quadratically smaller than $\norm{\cE(u_n)}$.
Unfortunately, our operator $\cE'(u)$ is hard to invert because it is the sum of a
constant coefficient difference operator and a multiplication operator.

Levi-Moser's trick consists in adding a term which does not affect the quadratic
gain but makes it possible to determine easily the increment~$h$:
let $u=u_n$ and $A=1+\pth u_n$; if, instead of requiring $\cE'(u)[h] = -\cE(u)$, we require
\begin{equation}	\label{eq:newlin}
A \cE'(u)[h] - h \cE'(u)[A] = - A \cE(u)
\end{equation}
and manage to get a solution~$h$ of size comparable to~$\norm{\cE(u)}$,
then we get $A \cE(u+h) = h \cE'(u)[A] + A Q(u,h)$, hence
\begin{equation}	\label{eq:newerr}
\cE(u+h) = \frac{h}{A} \pth\big( \cE(u) \big) + Q(u,h),
\end{equation}
which might be sufficient to ensure the convergence of the scheme.
Now, equation~\eqref{eq:newlin} is tractable because the multiplication operator
part in~$\cE'(u)$ cancels out from the \lhs:
equation~\eqref{eq:newlin} is equivalent to
\begin{equation}	\label{eq:newsimp}
A \De h - h \De A = A \cE(u)
\end{equation}
and Lemma~\ref{factorization} shows how to factorize the \lhs, while
Lemma~\ref{zeromean} shows that the \rhs\ has zero mean value, which turns out
to be sufficient to obtain a solution~$h$, as stated in Lemma~\ref{A12}.


\begin{lemma}\label{factorization}
For any $r>0$ and $A,h\in\Cpm r$ such that $A$ is invertible,
\begin{equation}	\label{eqfactoriz}
A \De h - h \De A = \na^-\Big( A A^+ \na\Big(\frac{h}{A}\Big)\Big)
%
%
\end{equation}
with the notation of Definition~\ref{defCRpCRm}.
\end{lemma}

\begin{proof}
Let $w = h/A$ and $a = A A^+$. Since $\ph\mapsto\ph^+$ and $\ph\mapsto\ph^-$ are
algebra maps, the \lhs\ of~\eqref{eqfactoriz} is
\[
A h^+ + A h^- - h A^+ - h A^- = a w^+ + a^- w^- - a w - a^- w
= a (\na w) - a^- (\na w)^-,
\]
whence the result follows.
%
%
\end{proof}


\begin{lemma}	\label{zeromean}
Let $u \in \Cpm r$ satisfy $\norm{u}_r < R_0 - r$, so that 
$\cE(u) = -\De u + \eps f \circ (\id +u) \in \Cn r$ is well-defined, 
and let $A = 1 + \pth u$.
Then $A \cE(u)$, which belongs to $\Cn{r'}$ for every $r' \in (0,r)$, has zero
mean value:
\[
\mean{A\cE(u)} =0.
\]
\end{lemma}

\begin{proof}
Since $\mean f = 0$, we can write~$f$ as the $\th$-derivative of a periodic function
$F \in H^\infty(S_{R_0})$.
The function $A \cdot \big( \eps f \circ (\id +u) \big)$ clearly has zero mean value, as the
$\th$-derivative of $\eps F \circ (\id +u)$.
As for the remaining part, adding and subtracting $(\pth u)^+ \na u$, we can
write it as
\[ \begin{split} 
- A \De u &= 
- \De u -(\pth u) (\na u - \na^- u)
= - \De u + \big( (\pth u)^+ - \pth u \big) \na u - (\pth u)^+ \na u + \pth u (\na u)^- \\
&= - \De u + (\na \pth u)\na u - \na \big( \pth u \cdot (\na^- u) \big)
= - \De u + \pth\big(\tfrac{1}{2}(\na u)^2\big) 
          - \na\big( \pth u \cdot (\na^- u) \big),
\end{split} \]
whence the claim follows since the difference operators $\De$ and $\na$ as well
as the differential operator $\pa_\th$ kill the constant terms.
\end{proof}


The next lemma shows how to find a solution of an equation like~\eqref{eq:newsimp} 
in the form $h = A w$, $w = \cF\big( A, \cE(u) \big)$, with a functional~$\cF$
which involves the operators~$\Ga$ and~$\Ga^-$ of Proposition~\ref{proplinlemma}
and the operator of division by $A A^+$.

\begin{lemma}	\label{A12}
Suppose $0 < r' < r$, $E\in\Cnr r$, $A\in\Cpmr r$ and $\nopm{A-1}{r} \le 1/6$.
Then there exists a unique $w \in \Cpmr{r'}$, which we denote $\cF(A,E)$, such
that $\mean w = 0$ and
\begin{equation}	\label{eqFsolves}
A \De(A w) - (A w) \De A = AE - \mean{AE}.
\end{equation}
Moreover, $\pth w\in \Cpmr{r'}$ as well and there exists a positive constant
$C_2$ which depends only on~$\tau$
such that
\begin{equation}	\label{eq:abounds}
\nopm{w}{r'} \le \frac{C_2 M^4}{(r-r')^{2\sig}} \norm{E}_r,
\quad
\nopm{\pth w}{r'} \le \frac{C_2 M^4}{(r-r')^{2\sig+1}} \norm{E}_r,
\end{equation}
where $\sig = 4 + 2\tau$.
\end{lemma}


\begin{proof}
Let $\de = \nopm{A-1}{r}$, thus $0\le \de \le 1/6$,
and $r''= (r'+r)/2$, whence $0 < r' < r'' < r$.
\begin{enumerate}
\item
We first observe that $\al \defeq \frac{1}{A A^+} \in \Cmr r$, with
\[
\nom{\al -1}{r} \le \frac{1}{(1-\de)^2}-1 \le \frac{11}{25} < \demi,
\]
because $\frac{1}{A}-1$ can be written as the absolutely convergent series 
$\sum_{n\ge1} (1-A)^n$ of $\Cpmr r$ with
$\nopm{\frac{1}{A}-1}{r} \le \frac{\de}{1-\de}$
and $\nom{\al-1}{r} \le \nom{\frac{1}{A}}{r} \nom{\frac{1}{A^+}-1}{r} 
+ \nom{\frac{1}{A}-1}{r}$.

As a consequence, if we denote simply by~$\norm{\cdot}$ the norm in
$\Cr = \Chol(K_M,\Br)$ or $\Crr$, we get
$\norm{1-\mean{\al}} \le \demi$, hence 
$\frac{1}{\mean\al} = \sum_{n\ge0} (1-\mean\al)^n$ absolutely convergent in $\Crr$ and
\[
\frac{1}{\mean\al} \in \Crr, \quad
\norm{\frac{1}{\mean\al}} \le 2.
\]

\item
Formula~\eqref{eqfactoriz} of Lemma~\ref{factorization} shows that
\[
\text{\eqref{eqFsolves}} \ens\Leftrightarrow\ens
\na^- \Big( \frac{1}{\al} \na w \Big) = AE - \mean{AE},
\]
where $AE\in\Cnr r$.
Now, in view of Proposition~\ref{proplinlemma} and Remark~\ref{remker},
\[
\text{\eqref{eqFsolves}} \ens\Leftrightarrow\ens
\exists \mu\in\Cr \;\text{such that}\; \frac{1}{\al} \na w = \psi + \mu,
\]
where $\psi =  \Ga^-(AE) \in \Cmr{r''}$.
\item
The equation $\na w = \al\psi + \mu\al$ leaves no choice but 
\[
\mu = \mu_0 \defeq - \frac{1}{\mean\al} \mean{\al\psi}
\]
(only possibility for having $\mean{\al\psi + \mu\al}=0$).
Notice that $\mu_0\in\Crr$. We end up with
\[
%
%
\big( \mean w = 0 \big) \; \& \; \text{\eqref{eqFsolves}} 
%
%
\ens\Leftrightarrow\ens w = \Ga\chi,
\]
where $\chi = \al\psi + \mu_0\al$. We notice that $\chi \in \Cmr{r''}$, thus, in
view of Remark~\ref{remGaGamphm}, the unique solution to our problem satisfies
\[
w \in \Cpmr{r'}, \quad
\nopm{w}{r'} \le \frac{C_1 M^2}{(r''-r')^{\sig}} \nom{\chi}{r''}, \quad
\nopm{\pa_\th w}{r'} \le \frac{C_1 M^2}{(r''-r')^{\sig+1}} \nom{\chi}{r''}.
\]
\item
We obtain the bounds~\eqref{eq:abounds} by observing that, on the one hand,
\[
\nom{\psi}{r''} \le \frac{7}{6} \frac{C_1 M^2}{(r-r'')^{\sig}} \nom{E}{r}
\]
and, on the other hand,
$\norm{\mu_0} \le 2 \nom{\al\psi}{r''} \le 3 \nom{\psi}{r''}$ 
and $\nom{\chi}{r''} \le \frac{3}{2} \nom{\psi+\mu_0}{r''}$,
thus $\nom{\chi}{r''} \le 7 \frac{C_1 M^2}{(r-r'')^{\sig}} \nom{E}{r}$
and the result follows with $C_2 = 7 C_1^2 \cdot 2^{2\sig+1}$.
\end{enumerate}
\end{proof}


Putting Lemma~\ref{zeromean} and Lemma~\ref{A12} together,
taking into account the fact that~\eqref{eq:newsimp} implies~\eqref{eq:newerr}
and working out the appropriate estimates,
we can summarize the inductive step in
\begin{proposition}	\label{B}
Suppose that $0<r'<r<R_0$, $r-r'\le1$, $\rho\le1$, and that $u\in \Cpmr r$ satisfies
$\pth u\in \Cpmr r$ and
\begin{equation}	\label{eq:hyp}
\nopm{u}{r} \le R_0-r, \qquad
%
%
\nopm{\pth u}{r} \le 1/6.
%
%
\end{equation}
Let $\sig = 4+2\tau$.
Then
\[
A \defeq 1+\pth u \in \Cpmr r
\quad\text{and}\quad
w \defeq \cF\big( A, \cE(u) \big) \in \Cpmr{r'}
\]
are well-defined and:
\begin{enumerate}
\item 
The function~$\pth w$ belongs to $\Cpmr{r'}$ and
\[
\nopm{Aw}{r'} \le 2(r-r')\xi, \quad
\nopm{\pth(Aw)}{r'} \le 4\xi, 
\quad \text{where} \quad
\xi \defeq \frac{C_2 M^4}{(r-r')^{2\sig+1}} \norm{\cE(u)}_r.
\]
%
\item
If $\nopm{u}{r}+2(r-r')\xi \le \demi(R_0-r')$ then $\cE(u+Aw) \in \Cnr{r'}$ and
%
%
\begin{equation}	\label{eq:stimaE}
\norm{\cE(u+Aw)}_{r'} \le \frac{C_3 M^8}{(r-r')^{4\sig}}
\big( \norm{\cE(u)}_r \big)^2,
\end{equation}
where $C_3$ is a positive constant which depends only on~$\tau$ and
$\norm{f''}_{R_0}$.
\end{enumerate}
\end{proposition}


\begin{proof}
By~\eqref{propcE}, we have a well-defined $\cE(u)\in\Cnr r$ and then, by
Lemma~\ref{A12}, a well-defined $w\in\Cpmr{r'}$ which satisfies
\[
\nopm{w}{r'} \le (r-r') \xi, \quad
\nopm{\pth w}{r'} \le \xi.
\]
%
%
Since $\nopm{A}{r} \le 7/6 < 2$, we get
$\nopm{Aw}{r'} \le 2(r-r')\xi$
%
%
and the Cauchy inequalities yield 
$\nopm{\pth A}{r'} \le \frac{1}{r-r'}\nopm{A}{r}$,
whence
\[ 
\nopm{\pth(Aw)}{r'} \le \nopm{A}{r} \left(
\frac{1}{r-r'}\nopm{w}{r'}+\nopm{\pth w}{r'} \right)
\le 4\xi. 
%
%
%
\]
%


Supposing now that $\nopm{u}{r}+2(r-r')\xi \le \demi(R_0-r')$, 
%
%
we have a well-defined $\cE(u+Aw) \in \Cnr{r'}$ still by~\eqref{propcE},
but the fact that $\mean{A\cE(u)}=0$ (Lemma~\ref{zeromean}) implies that $h=Aw$
solves~\eqref{eq:newsimp}, we can thus take advantage of~\eqref{eq:newerr},
which takes the form
\begin{equation}	\label{eq:erruAw}
\cE(u+Aw) = w \pth\big(\cE(u)\big) + 
(Aw)^2 \int_0^1 \eps f''\circ(\id+ u + t Aw)(1-t)\, \dd t.
\end{equation}
Since $\norm{u+t Aw}_{r'} \le \demi(R_0-r')$, by Lemma~\ref{lemcompos},
we have a continuous curve $t\in[0,1] \mapsto
\eps f''\circ(\id+ u + t Aw) \in \Cnr{r'}$, bounded in norm by $2\rho\norm{f''}_{R_0}$
and the second term in the \rhs\ of~\eqref{eq:erruAw} is bounded in norm by
$8(r-r')^2\xi^2\rho\norm{f''}_{R_0}$.
For the first term, we use the cauchy inequalities:
$\norm{w \pth\big(\cE(u)\big)}_{r'} \le (r-r')\xi \cdot \frac{1}{r-r'}
\norm{\cE(u)}_r$.
Using $\rho\le 1$, $M\ge1$ and $r-r'\le 1$, we get the desired bound with
$C_3 \defeq C_2 + 8 C_2^2 \norm{f''}_{R_0}$.
\end{proof}


To simplify further the estimates, we can take into account the fact
that $2\sig+1 < 4\sig$ and $M>1$ (because $\tau>0$ and $M>2\ze(1+\tau)$), thus
$\frac{M^4}{(r-r')^{2\sig+1}} < \frac{M^8}{(r-r')^{4\sig}}$
and,
setting $C \defeq \max(4C_2,C_3)$,
content ourselves with
\begin{corollary}	\label{cor:rough}
Suppose $0<r'<r<R_0$, $r-r'\le1$, $\rho\le1$ and 
$\ka \le \min(R_0-r,\frac{1}{3})$. Then
\begin{multline}	\label{ineq:rough}
\nopm{u}{r}, \nopm{\pth u}{r} \le \de \le \frac{\ka}{2}
\quad\Rightarrow\quad
w \defeq \cF\big( A, \cE(u) \big) \in \Cpmr{r'} \ens\text{and}
\\
\nopm{Aw}{r'}, \nopm{\pth(Aw)}{r'} \le 
\frac{C M^8}{(r-r')^{4\sig}} \norm{\cE(u)}_r
\end{multline}
(still with $A \defeq 1+\pth u$)
and, if the \rhs\ in~\eqref{ineq:rough} is $\le \frac{\ka}{2}-\de$, then
\begin{equation}	\label{ineq:errquadr}
\norm{\cE(u+Aw)}_{r'} \le 
\frac{C M^8}{(r-r')^{4\sig}} \big( \norm{\cE(u)}_r \big)^2.
\end{equation}
\end{corollary}


\subsection{The iterative scheme}	\label{secCVscheme}


Following \cite{LM}, as alluded to at the end of Section~\ref{sec:reduc}, we now
prove Theorem~\ref{thmtiuCunhol} by means of the modified Newton method on a
scale of Banach spaces corresponding to strips $S_{R_n}$ which go on shrinking
around a limiting strip $S_{R_\infty}$.

We suppose that we are given $0<R<R_0$ and $f\in H^\infty(S_{R_0})$ real
analytic with $f''\in H^\infty(S_{R_0})$.
Without loss of generality, we also suppose $R_0-R \le \frac{4}{3}$.
We set
\[
R_\infty \defeq \frac{R+R_0}{2}, \qquad
\ka \defeq \frac{R_0-R_\infty}{2} = \frac{R_\infty-R}{2} \le \frac{1}{3},
\qquad R_n \defeq R_\infty + 2^{-n-1}\ka \quad\text{for $n\ge1$}.
\]
We observe that $R_0 > R_n > R_{n+1}$,
\[
R_n-R_{n+1} = 2^{-n} \ka < 1
\]
and
$R_0-R_n \ge R_0-R_1 = \ka$ for $n\ge1$.
%
%
Corollary~\ref{cor:rough} with $r=R_n$ and $r'=R_{n+1}$ motivates

\begin{lemma}	\label{lemepsn}
Suppose $0 < \eps_1 \le \dfrac{\ka^{4\sig}}{2^{8\sig+1} C M^8}$.
Then the induction
$
\eps_{n+1} \defeq 2^{4\sig n}\ka^{-4\sig} C M^8 (\eps_n)^2
$
determines a sequence of positive numbers which satisfies
$\dst\sum_{n\ge1} 2^{4\sig n}\ka^{-4\sig} C M^8 \eps_n 
\le \frac{\ka}{2}$.
\end{lemma}


\begin{proof}
Let $\bar\eps_n \defeq 2^{4\sig(n+1)}\ka^{-4\sig} C M^8 \eps_n$ for $n\ge1$,
so that $\bar\eps_{n+1} = (\bar\eps_n)^2$
and $\bar\eps_1 = 2^{8\sig}\ka^{-4\sig} C M^8 \eps_1 \le \demi$. 
We have
$\bar\eps_n = (\bar\eps_1)^{2^{n-1}} \le (\bar\eps_1)^n \le 2^{-(n-1)} \bar\eps_1$
and the result follows from 
$\sum_{n\ge1}\bar\eps_n \le 2 \bar\eps_1 
< 2^{4\sig}\frac{\ka}{4}$.
\end{proof}


Let $0 < \rho \le1$.
We set $u_1 \defeq 0 \in \Cpmr{R_1}$ and $\de_1 \defeq 0$, so that 
$\cE(u_1) = \eps f \in \Cnr{R_1}$ with 
$\norm{\cE(u_1)}_{R_1} = \rho \norm{f}_{R_1}$.
From now on, we suppose that
\begin{equation}	\label{ineqrhoMhuit}
\eps_1 \defeq \rho \norm{f}_{R_1} \le \dfrac{\ka^{4\sig}}{2^{8\sig+1} C M^8}
\end{equation}
and we define inductively
\begin{equation}
\de_{n+1} \defeq \de_n + 2^{4\sig n}\ka^{-4\sig} C M^8 \eps_n, \quad
\eps_{n+1} \defeq 2^{4\sig n}\ka^{-4\sig} C M^8 (\eps_n)^2,
\end{equation}
so that, by Lemma~\ref{lemepsn}, $\de_n \le \frac{\ka}{2}$ for all $n\ge1$.
This is sufficient to apply inductively Corollary~\ref{cor:rough}:
the formula
\begin{equation}
u_{n+1} \defeq u_n + (1+\pth u_n) \cF\big( 1+\pth u_n, \cE(u_n) \big),
\qquad n\ge1
\end{equation}
yields a sequence of functions $u_n \in \Cpmr{R_n}$ satisfying
\[
\nopm{u_n}{R_n}, \nopm{\pth u_n}{R_n} \le \de_n, \qquad
\norm{\cE(u_n)}_{R_n} \le \eps_n.
\]
Moreover, $\nopm{u_{n+1}-u_n}{R_\infty} \le \nopm{u_{n+1}-u_n}{R_{n+1}} 
\le 2^{4\sig n}\ka^{-4\sig} C M^8 \eps_n$,
hence there is a limit 
\[ u_\infty \defeq \lim_{n\to\infty} 
{u_n}_{| K_M \times S_{R_\infty} \times \DD_\rho}
\in \Cpmr{R_\infty} \]
with $\nopm{u_\infty}{R_\infty} \le \frac{\ka}{2} < \min(R_0-R_\infty,R_\infty-R)$.

This limit~$u_\infty$ is indeed a solution of the equation $\cE(u)=0$ because $\eps_n
\xrightarrow[n\to\infty]{} 0$ and $\cE$ is Lipschitz when viewed as a map from 
$\{\, u \in \Cpmr{R_\infty} \mid \nopm{u}{R_\infty} \le R_0-R_\infty \,\}$
to $\Cnr{R_\infty}$
(the Lipschitz constant is $\le 4 + \rho\norm{f'}_{R_0}$).



\section{The unit circle as a natural boundary}	\label{secrat} 


This section is devoted to the proof of Theorem~\ref{thmrat}.
It will be a consequence of 
\begin{proposition}	\label{proprat}
Let $\om_* = \frac{p}{m}$ with $p\in\ZZ$ and $m\in\NN^*$.
Suppose that $f$ is a non-zero trigonometric polynomial with zero mean value, or that $f$ is an
analytic function on~$\TT$ with zero mean value whose Fourier expansion contains
at least one non-zero coefficient with index $k\in m\ZZ$.
Then the difference equation
\begin{equation}	\label{equrat}
u(\th+\om_*,\eps) -2 u(\th,\eps) + u(\th-\om_*,\eps) = \eps f\big( \th + u(\th,\eps) \big)
\end{equation}
has no solution formal in~$\eps$ and analytic~$\th$,
\ie no formal solution $\dst u = \sum_{n\ge1}\eps^n u_n(\th)$ with all coefficients~$u_n$
analytic on~$\TT$.
\end{proposition}


\begin{proof}[Proposition~\ref{proprat} implies Theorem~\ref{thmrat}]
%
%
Suppose that the assumptions of Theorem~\ref{thmrat} hold and that the
restriction of~$\ti u_M$ to~$\ovi{K}\str{1.55}\IN_M$ 
(or its restriction to~$\ovi{K}\str{1.55}\EX_M$)
has an analytic continuation in an open connected
set~$U$ which intersects the unit circle~$\SS$ and~$\ovi{K}\str{1.55}\IN_M$ 
(or~$\ovi{K}\str{1.55}\EX_M$);
%
denote this analytic continuation $u \in \gO(U, \BB_{R,\rho})$.
We shall reach a contradiction.

Let us choose an open disc $D\subset\CC$ such that $E(D)\subset U$ and
$D\cap\RR\neq\emptyset$.
Let us take $M_1>M$ large enough so that $D\cap A_{M_1}^\RR \neq\emptyset$
(this is possible since $D\cap\RR$ is an open interval and the Lebesgue measure
of $(D\cap\RR)-A_{M_1}^\RR$ tends to~$0$ as $M_1\to\infty$).
We observe that, perhaps at the price of diminishing~$\rho$, we can consider~$u$ as
an analytic continuation of the restriction of~$\ti u_{M_1}$
to~$\ovi{K}\str{1.55}\IN_M$ (or~$\ovi{K}\str{1.55}\EX_M$) as well.
Therefore, for each $\om\in D\cap A_{M_1}^\RR$ and $\eps\in(-\rho,\rho)$, the function $\th \mapsto
u(\ee^{2\pi\I\om})(\th,\eps)$ is a solution of the difference equation~\eqref{eqseconddiffuom};
moreover $u(\ee^{2\pi\I\om})(\th,0) \equiv 0$ (because, for $\om$ irrational,
\eqref{eqseconddiffuom} implies $u(\ee^{2\pi\I\om})_{|\eps=0}$ constant and we
have normalized~$\ti u_{M_1}$ by imposing zero mean value).
Since $D\cap A_{M_1}^\RR$ is not discrete, by analytic continuation \wrt~$\om$,
we also have $\th\in\TT \mapsto u(\ee^{2\pi\I\om})(\th,\eps)$ analytic solution
of~\eqref{eqseconddiffuom} for each $\om\in D\cap\RR$ and $\eps\in(-\rho,\rho)$,
with $u(\ee^{2\pi\I\om})(\th,0) \equiv 0$.

Choose $m_*\in\NN^*$ larger that $1/|D\cap\RR|$: for every $m\ge m_*$ there is
a $p_m\in\ZZ$ such that $\frac{p_m}{m} \in D\cap\RR$.
If $f$ is not a trigonometric polynomial, we then take $m\ge m_*$ such that
the $m$th Fourier coefficient of~$f$ is non-zero;
if $f$ is a trigonometric polynomial, we take any $m\ge m_*$.
In both cases, equation~\eqref{eqseconddiffuom} with $\om = \frac{p_m}{m}$ shows
that the Taylor expansion of $u(\ee^{2\pi\I \frac{p_m}{m}}) \in \BB_{R,\rho}$
\wrt~$\eps$ is a formal solution of~\eqref{equrat} with coefficients analytic
on~$\TT$
and Proposition~\ref{proprat} yields a contradiction.
\end{proof}


\medskip

\noindent\emph{Proof of Proposition~\ref{proprat}.}
Let $\om_*$ and~$f$ be as in the assumptions of Proposition~\ref{proprat}.
Let $\gC$ denote the differential $\CC$-algebra~$\gC^\om(\TT)$ of all
analytic functions of~$\TT$.
Equation~\eqref{equrat} can be written
\begin{equation}	\label{equratbis}
\De_* u = \eps f \circ(\id+u),
\end{equation}
where the \rhs\ is defined for any $u\in\eCe$, \ie any formal series
in~$\eps$ with coefficients analytic on~$\TT$ and without $0$th order term, by
the Taylor formula
\begin{equation}	\label{eqfTaylor}
f\circ(\id+u) = \sum_{r\ge0} \frac{1}{r!} f\suppar r u^r
\in \Ce
\end{equation}
(formally convergent for the complete metric space structure induced by the
$\eps-$adic valuation in $\Ce$),
while the \lhs\ involves the $\CC[[\eps]]$-linear operator
\[ \De_* \col \Ce \to \Ce \]
which extends the second-order difference operator
\[
\De_* \col \gC \to \gC, \qquad
(\De_*\ph)(\th) \defeq \ph(\th+\om_*) - 2\ph(\th) + \ph(\th-\om_*).
\]
To prove Proposition~\ref{proprat}, we argue by contradiction and suppose that
we are given a formal solution $u = \sum_{n\ge1} \eps^n u_n\in\eCe$ of
equation~\eqref{equratbis}.


Let us view~$\gC$ as a linear subspace of the Hilbert space $L^2(\TT)$.
The Fourier transform associates with any~$\ph$ its coordinates
$(\hat\ph_k)_{k\in\ZZ}$, $\hat\ph_k = \cF_k(\ph)$, on the Hilbert basis
$(e_k)_{k\in\ZZ}$ (notation of Section~\ref{secFourier}).
We have an orthogonal decomposition
\[
\gC = V_0 \oplus V_1 \oplus \cdots \oplus V_{m-1},
\]
where $V_j \defeq \{\, \ph\in\gC \mid \hat\ph_k = 0 \;\text{for}\; k\notin
j+m\ZZ \,\}$; let us denote by $\Pi_0,\ldots,\Pi_{m-1}$ the corresponding
orthogonal projectors.
For instance, the Fourier expansions of~$f$ and~$\Pi_0 f$ read 
\begin{equation}	\label{eqFourierf}
f = \sum_{k\in\ZZ^*} \hat f_k e_k, \quad
\Pi_0 f = \sum_{k\in m\ZZ^*} \hat f_k e_k.
\end{equation}
The spectral decomposition of~$\De_*$ is obtained from the relations
$\De_* e_k = D_k e_k$, $k\in\ZZ$, with
\begin{equation}	\label{eqdefDk}
D_k = \ee^{2\pi\I k\om_*} - 2 + \ee^{-2\pi\I k\om_*} 
= - 4 \sin^2(k\pi\om_*).
\end{equation}
Since $D_k$ depends only on $k+m\ZZ$, we get
\begin{equation}	\label{eqdecDest}
\De_* = D_0 \Pi_0 + D_1 \Pi_1 + \cdots + D_{m-1} \Pi_{m-1}.
\end{equation}
Observe that $D_0=0$ and $D_1,\ldots,D_{m-1}<0$. Therefore
\begin{equation}	\label{eqkerrange}
\ker(\De_*) = V_0, \quad \range(\De_*) = V_0^\perp = V_1 \oplus \cdots \oplus V_{m-1}.
\end{equation}


This is enough to conclude the proof when $f$ is not a trigonometric polynomial:
indeed, our assumption on the Fourier coefficients of~$f$ in this case together
with~\eqref{eqFourierf} imply that $\Pi_0 f$ is not identically zero, hence
$f\notin \range(\De_*)$,
whereas expanding equation~\eqref{equratbis} in powers of~$\eps$ yields
$\De_* u_1 = f$, a contradiction.


From now on, we thus suppose that $f$ is a trigonometric polynomial of the form
\begin{equation}	\label{eqassumA}
f = \sum_{|k|\le K} \hat f_k e_k,\qquad \hat f_K = A \neq 0,
\end{equation}
with a certain $K\in\NN^*$
(the case $\hat f_{-K}\neq0$ is reduced to~\eqref{eqassumA} by changing~$\th$
into~$-\th$).
It is sufficient to reach a contradiction in this case.


Let us first check that one can suppose that all the coefficients~$u_n$ of~$u$
belong to~$V_0^\perp$.
\begin{lemma}	\label{lemreduc}
For any solution $u \in \eCe$ of~\eqref{equratbis}, the formal series $\id+\Pi_0
u$ has a composition inverse of the form $\id+a$ with $a\in\eps V_0[[\eps]]$ 
and the formula 
\[
u_* = (u - \Pi_0 u) \circ (\id+a)
\]
defines a formal solution of equation~\eqref{equratbis} which has all its
coefficients in~$V_0^\perp$.
\end{lemma}

\begin{proof}[Proof of Lemma~\ref{lemreduc}]
Of course, composition is to be understood ``\wrt~$\th$ at fixed~$\eps$'', as
in~\eqref{eqfTaylor}, \ie we are inverting
$(\eps,\th) \mapsto \big(\eps, \th + (\Pi_0 u)(\th,\eps) \big)$.
The $0$th order coefficient of~$\Pi_0 u$ being zero, $\id+\Pi_0 u$ has a
composition inverse which is given by the (formally convergent) Lagrange
inversion formula
\[
(\id + \Pi_0 u)\ii = \id + a, \qquad
a = \sum_{s\ge1} \frac{(-1)^s}{s!} 
\Big(\frac{\dd\,}{\dd\th}\Big)^{s-1}\big[ \big(\Pi_0 u\big)^s \big]
\in \eCe.
\]
Since~$V_0$ consists of all $\frac{1}{m}$-periodic analytic functions of~$\th$,
it is a differential subalgebra of~$\gC$ and each subspace~$V_j$ is a $V_0$-module;
since $\Pi_0 u\in\eps V_0[[\eps]]$, we deduce that $a\in\eps V_0[[\eps]]$, that
the composition operator $\ph\mapsto\ph\circ(\id+a)$ leaves $V_j[[\eps]] \subset
\Ce$ invariant for each~$j$ 
(beware that $V_j$ is not a ring for $j\neq0$, thus $V_j[[\eps]]$ is just a
linear subspace of $\Ce$, in fact a $V_0[[\eps]]$-submodule)
and that
\begin{equation}	\label{eqDestmod}
\De_*\big( \ph\circ(\id+a) \big) = (\De_*\ph)\circ(\id+a).
\end{equation}
The composition by $\id+a$ leaves also $V_0^\perp[[\eps]]$ invariant, hence $u_* =
(u-\Pi_0 u)\circ(\id+a)$ is a well-defined formal series with all its
coefficients in~$V_0^\perp$ and without $0$th order term;
the fact that it is a solution of~\eqref{equratbis} stems from~\eqref{eqDestmod},
the associativity of composition and the relations
$\De_*( u - \Pi_0 u ) = \De_* u = \eps f \circ(\id +u)$
and
$(\id+u)\circ(\id+a) = (\id+\Pi_0 u + u-\Pi_0 u)\circ(\id+a) = \id + u_*$.
\end{proof}


\begin{lemma}	\label{lemdefE}
The formula
\[
E = \la_1 \Pi_1 + \cdots + \la_{m-1} \Pi_{m-1},
\]
with 
$\la_j \defeq -\frac{1}{4\sin^2(j\pi\om_*)} < 0$ for $j=1,\ldots,m-1$,
defines an operator $E \col \gC \to \gC$ such that, for any $\ph,\psi\in\gC$,
\[
\De_*\ph = \psi \ens\text{and}\ens \ph\in V_0^\perp
\quad\Leftrightarrow\quad
\psi \in V_0^\perp \ens\text{and}\ens \ph = E\psi.
%
%
%
\]
\end{lemma}


\begin{proof}[Proof of Lemma~\ref{lemdefE}]
Immediate consequence of~\eqref{eqdefDk}--\eqref{eqdecDest}.
\end{proof}


\begin{proof}[End of the proof of Proposition~\ref{proprat}]
We assume that~$f$ is of the form~\eqref{eqassumA} and we have a formal
solution~$u$ of equation~\eqref{equratbis}.
By Lemma~\ref{lemreduc}, at the price of replacing~$u$ with~$u_*$, we can
suppose that all the coefficients~$u_n$ of~$u$ belong to~$V_0^\perp$.
Let
\begin{equation}	\label{eqdefg}
g = \eps f\circ(\id+u) = \sum_{n\ge1} \eps^n g_n, \qquad g_n\in\gC.
\end{equation}
Lemma~\ref{lemdefE} allows us to rewrite~\eqref{equratbis} as
\begin{equation}	\label{eqrewrite}
g_n\in V_0^\perp, \quad u_n = E g_n, \qquad n\ge1.
\end{equation}
A simple computation yields
\[
g_n = \sum_{|k|\le nK} \cF_k(g_n) e_k, \quad
u_n = \sum_{|k|\le nK} \cF_k(u_n) e_k,
\qquad n\ge1,
\]
with $\ga_n \defeq \cF_{nK}(g_n)$ and $\al_n \defeq \cF_{nK}(u_n)$ inductively
determined by
$\ga_1 = A$ and
\begin{align*}
\al_n &= \la\subcar{nK} \ga_n, & n\ge1, \\[1ex]
\ga_n &= \sum_{r=1}^{n-1} \frac{(2\pi\I K)^r A}{r!} 
\sum_{\substack{n_1,\ldots,n_r\ge1 \\ n_1+\cdots+n_r = n-1}} 
\al_{n_1} \cdots \al_{n_r},
& n\ge2,
\end{align*}
with the notation $\la\subcar k\defeq0$ if $k\in m\ZZ$ 
and $\la\subcar k \defeq \la_j$ with
$j\in\{1,\ldots,m-1\}$ such that $k\in j+ m\ZZ$
if $k\notin m\ZZ$.

Defining inductively $\be_1 \defeq 1$ and
\begin{align*}
\be_n &= \sum_{r=1}^{n-1} \frac{1}{r!} 
\sum_{\substack{n_1,\ldots,n_r\ge1 \\ n_1+\cdots+n_r = n-1}} 
(-\la\subcar{n_1 K}) \be_{n_1} \cdots (-\la\subcar{n_r K}) \be_{n_r},
\qquad n\ge2, \\
\intertext{we get}
\ga_n &= (-2\pi\I K)^{n-1} A^n \be_n, \qquad n\ge1.
\end{align*}
Let~$n_*$ denote the smallest integer in $\{1,\ldots,m\}$ such that $n_* K \in m \ZZ$.
We see that $\be_1,\ldots,\be_{n_*}>0$.
Since $A\neq0$, we reach a contradiction when comparing the requirement
$g_{n_*}\in V_0^\perp$ (according to~\eqref{eqrewrite}) and the relation
$\cF_{n_* K}(\Pi_0 g_{n_*}) = \cF_{n_* K}(g_{n_*}) = \ga_{n_*} \neq0$
which implies $\Pi_0 g_{n_*} \neq0$.
\end{proof}


\begin{rem}
When $f$ is assumed to be a real analytic function of zero mean value which does
not belong to~$V_0$ (\ie there exists $k\in m\ZZ$ such that the $k$th Fourier
coefficient of~$f$ is non-zero), one can prove by a more geometric method that
there exists $\rho_* = \rho_*(m,f)>0$ such that, for any non-zero $\eps\in(-\rho_*,\rho_*)$,
the difference equation~\eqref{equrat} has no solution~$u$ real analytic
on~$\TT$
(observing that, associated with such a solution, there would be a curve
$\ga(\th) = \big( \th + u(\th), \om_* + u(\th) - u(\th-\om_*) \big)$
consisting of fixed points of~$T_\eps^m$,
and showing that these fixed points are isolated).
\end{rem}


\vspace{1.5cm}


\noindent {\em Acknowledgements.}
The authors thank H.~Eliasson, P. Lochak and J.-C.~Yoccoz for interesting discussions.
The authors acknowledge the support of the Centro di Ricerca Matematica Ennio de
Giorgi.
They also wish to thank the Max Planck Institute for Mathematics and the
organizers of the 2009 program ``Dynamical Numbers''.
The research leading to these results has received funding from the
European Comunity's Seventh Framework Program (FP7/2007--2013) under Grant
Agreement n.~236346.


\vfill\eject


\bigskip

\noindent
\textbf{Carlo Carminati}

\noindent
Dipartimento di Matematica, Universit\`a di Pisa;
Largo Bruno Pontecorvo 5, 56127 Pisa, Italy\\
E-mail: \url{carminat@dm.unipi.it}

\medskip

\noindent
\textbf{Stefano Marmi}

\noindent
Scuola Normale Superiore; Piazza dei Cavalieri 7, 56126 Pisa, Italy\\
E-mail: \url{s.marmi@sns.it}

\medskip

\noindent
\textbf{David Sauzin}

\noindent
Scuola Normale Superiore di Pisa; Piazza dei Cavalieri 7, 56126 Pisa, Italy\\
E-mail: \url{sauzin@imcce.fr}\\
(On leave from Institut de m\'ecanique c\'eleste, CNRS; 77 av.\ Denfert-Rochereau, 75014 Paris, France)


\end{document}